\documentclass[journal]{IEEEtran}
\ifCLASSINFOpdf
\else
\fi
\usepackage{url}

\usepackage{graphics} 
\usepackage{graphicx}
\usepackage{times} 
\usepackage{amsmath} 
\usepackage{amssymb}  
\usepackage{algorithm}
\usepackage{color}

\usepackage{verbatim}

\usepackage{amsthm}

\usepackage{mathtools}
\usepackage{bbm}

\usepackage{subfigure}

\newtheorem{theorem}{\bf Theorem}  
\newtheorem{lemma}{\bf Lemma} \newtheorem{remark}{\bf Remark}
  \newtheorem{proposition}{\bf Proposition} 
\newtheorem{assumption}{\bf Assumption}  
\newtheorem{Algorithm}{\bf Algorithm}

\newcommand\bbr{\mathbb{R}}

\newcommand\calZ{\mathcal{Z}}

\newcommand\rma{\mathrm{a}}
\newcommand\rmb{\mathrm{b}}
\newcommand\rmc{\mathrm{c}}

\newcommand\rmf{\mathrm{f}}
\newcommand\rml{\mathrm{l}}
\newcommand\rms{\mathrm{s}}
\newcommand\rmr{\mathrm{r}}
\newcommand\rmu{\mathrm{u}}

\newcommand\rmy{\mathrm{y}}

\newcommand\rmA{\mathrm{A}}
\newcommand\rmB{\mathrm{B}}
\newcommand\rmV{\mathrm{V}}

\newcommand\rmeq{\mathrm{eq}}
\newcommand\rmlin{\mathrm{Lin}}
\newcommand\rmsp{{\rms}\text{$'$}}

\begin{document}
%
\title{Linear tracking MPC for nonlinear systems\\Part I: The model-based case}
%
%
%

\author{Julian Berberich$^1$, 
		Johannes K\"ohler$^{1,2}$, 
		Matthias A. M\"uller$^3$, 
		and Frank Allg\"ower$^1$.
		\thanks{Funded by Deutsche Forschungsgemeinschaft (DFG, German Research Foundation) under Germany's Excellence Strategy - EXC 2075 - 390740016 and under grant 468094890. 
		We acknowledge the support by the Stuttgart Center for Simulation Science (SimTech).
This project has received funding from the European Research Council (ERC) under the European Union’s Horizon 2020 research and innovation programme (grant agreement No 948679).
The authors thank the International Max Planck Research School for Intelligent Systems (IMPRS-IS)
for supporting Julian Berberich.}
\thanks{$^1$University of Stuttgart, Institute for Systems Theory and Automatic Control, 70550 Stuttgart, Germany (email:$\{$julian.berberich, frank.allgower$\}$@ist.uni-stuttgart.de)}
\thanks{$^2$Institute for Dynamical Systems and Control, ETH Zurich, ZH-8092, Switzerland (email:jkoehle@ethz.ch)}
\thanks{$^3$Leibniz University Hannover, Institute of Automatic Control, 30167 Hannover, Germany (e-mail:mueller@irt.uni-hannover.de)}}

%
%

\markboth{}%
{}
%



\IEEEoverridecommandlockouts

\IEEEpubid{\begin{minipage}{\textwidth}\ \\[12pt] \\ \\
\copyright 2021 IEEE. Personal use of this material is permitted. Permission from IEEE must be obtained for all other uses, in any current or future media, including reprinting/republishing this material for advertising or promotional purposes, creating new collective works, for resale or redistribution to servers or lists, or reuse of any copyrighted component of this work in other works.
\end{minipage}}

\maketitle

\begin{abstract}
We develop a tracking model predictive control (MPC) scheme for nonlinear systems using the linearized dynamics at the current state as a prediction model.
Under reasonable assumptions on the linearized dynamics, we prove that the proposed MPC scheme exponentially stabilizes the optimal reachable equilibrium w.r.t. a desired target setpoint.
Our theoretical results rely on the fact that, close to the steady-state manifold, the prediction error of the linearization is small and hence, we can slide along the steady-state manifold towards the optimal reachable equilibrium.
The closed-loop stability properties mainly depend on a cost matrix which allows us to trade off performance, robustness, and the size of the region of attraction.
In an application to a nonlinear continuous stirred tank reactor, we show that the scheme, which only requires solving a convex quadratic program online, has comparable performance to a nonlinear MPC scheme while being computationally significantly more efficient.
Further, our results provide the basis for controlling nonlinear systems based on data-dependent linear prediction models, which we explore in our companion paper.
\end{abstract}


%
\IEEEpeerreviewmaketitle
\section{Introduction}

Model predictive control (MPC) is a successful modern control technique, mainly due to its ability to cope with nonlinear dynamics, hard input and state constraints, and dynamic reference trajectories~\cite{rawlings2020model}.
The key idea of MPC centers around the repeated solution of an open-loop optimal control problem.
While its applicability to nonlinear systems is one of the main advantages of MPC, this requires solving a non-convex optimization problem online.
In order to cope with limited computational resources, usually only a limited number of iterations are performed~\cite{diehl2005nominal,diehl2009efficient,liao2020time,zanelli2021lyapunov}.
Further, most nonlinear MPC schemes require a (globally) accurate model of the underlying dynamical system and, if this is not available, techniques such as robust nonlinear MPC need to be employed, possibly further increasing the computational complexity~\cite{mayne2011tube}.

In this paper, we propose a tracking MPC scheme for nonlinear systems using the linearized dynamics at the current state as a prediction model.
We consider a tracking MPC formulation to steer the system to a desired target setpoint, which may potentially change online and can be unreachable by the system dynamics and constraints, similar to~\cite{limon2018nonlinear,koehler2020nonlinear}.
The cost function contains a tracking cost w.r.t.\ an artificial equilibrium as well as a penalty of 
the deviation between this artificial equilibrium and the desired target setpoint.
We prove that, if the weight matrix of this penalty is chosen sufficiently small, then the linearization-based prediction model is sufficiently accurate such that the closed loop converges to the optimal reachable equilibrium.
The latter is the closest possible equilibrium to the target setpoint, in case this setpoint is not an equilibrium of the system and / or does not satisfy the constraints.
Figure~\ref{fig:manifold} illustrates the main idea:
At time $t=0$, the open-loop predictions reach an (artificial) equilibrium output $y^{\rms*}(0)$ for the linearized system dynamics lying in a neighborhood of the nonlinear equilibrium manifold $\mathcal{Z}_{\rmy}^\rms$.
At time $t$, after repeated application of the proposed MPC scheme, the artificial equilibrium $y^{\rms*}(t)$ is closer to $y^{\rms\rmr}$, the optimal reachable equilibrium given the (potentially unreachable) setpoint reference $y^\rmr$.
The model used for prediction at time $t$ relies on the linearized system dynamics at the current state $x_t$.
By repeated application of the proposed MPC scheme, the artificial equilibrium $y^{\rms*}(t)$ slides along the manifold $\mathcal{Z}_{\rmy}^\rms$ towards $y^{\rms\rmr}$ and eventually, the closed-loop output converges to $y^{\rms\rmr}$.

\vspace{-10pt}

\begin{figure}[h!]
\begin{center}
\includegraphics[width=0.48\textwidth]{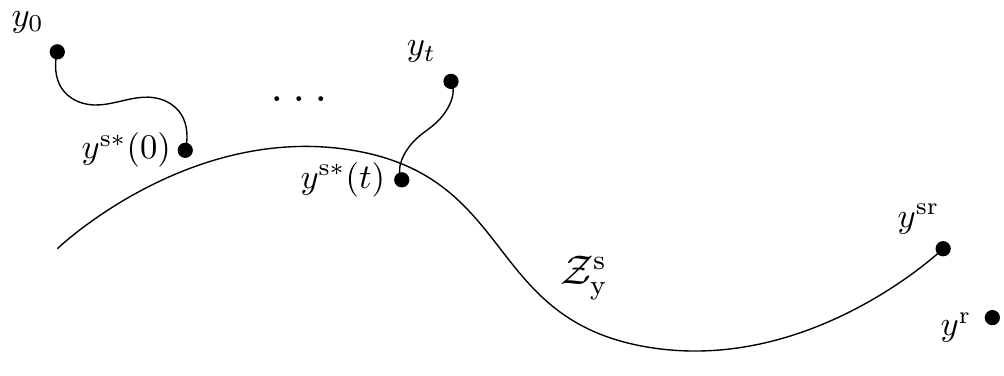}
\vspace{-11pt}
\end{center}
\caption{Scheme illustrating the basic idea of our approach.
The figure displays the output equilibrium manifold $\mathcal{Z}_{\rmy}^\rms$, the closed-loop output and artificial equilibrium at time $0$, the closed-loop output and artificial equilibrium at time $t$, the optimal reachable equilibrium $y^{\rms\rmr}$, and the setpoint reference $y^\rmr$.}
\label{fig:manifold}
\end{figure}

\vspace{-5pt}

The main contribution of this paper is to prove that the proposed MPC scheme exponentially stabilizes $y^{\rms\rmr}$ under reasonable assumptions on the linearized dynamics.
A key advantage of our scheme if compared to nonlinear tracking MPC schemes~\cite{limon2018nonlinear,koehler2020nonlinear} lies in its computational efficiency, since only a convex quadratic program (QP) needs to be solved online.
Moreover, the implementation only requires an accurate linear model around the steady-state manifold, which \newpage 
\noindent is easier to obtain than a complex nonlinear model.
Finally, this paper provides the basis for the results in our companion paper~\cite{berberich2021linearpart2} where we merge the idea of linearization-based tracking MPC with recent work on data-driven MPC~\cite{coulson2019deepc,berberich2021guarantees} to develop a data-driven MPC scheme for unknown nonlinear systems with closed-loop stability guarantees.

The real-time iteration scheme~\cite{diehl2005nominal,zanelli2021lyapunov} relies on a sequential approximation of the optimal control problem, resulting in predictions based on linear time-varying (LTV) models, and is thus conceptually similar to our work.
An LTV prediction model is also employed in~\cite{cannon2011robust}, where a tube around the previous candidate solution is used to reduce the linearization error.
The relation of our approach to existing LTV-based MPC methods for nonlinear systems will be discussed in more detail later in the paper.
Further related approaches are~\cite{lawrynczuk2007computationally,papadimitriou2020control} which estimate (approximate) LTV models of nonlinear systems from data and employ these in MPC schemes.
While these approaches lead to convex optimization problems which can be efficiently solved, they do not provide guarantees on closed-loop stability.
Another line of research to control nonlinear systems via MPC with linear prediction models relies on Koopman operator theory~\cite{korda2018linear}.
The key differences to our approach are that we exploit local linearity properties whereas the (infinite-dimensional) Koopman operator provides a globally valid linear model, and that we provide desirable closed-loop stability guarantees.

The paper is structured as follows.
In Section~\ref{sec:setting}, we state the problem setting and the key assumptions required for our theoretical results.
Next, we present our linearization-based MPC scheme for nonlinear systems in Section~\ref{sec:MPC}, and we prove theoretical properties such as recursive feasibility and closed-loop exponential stability.
In Section~\ref{sec:example}, we apply the scheme to a nonlinear continuous stirred tank reactor (CSTR).
Finally, Section~\ref{sec:conclusion} contains a summary and conclusion.

\subsubsection*{Notation}
Denote the set of all nonnegative integers by $\mathbb{I}_{\geq0}$ and the integers in the interval $[a,b]$ by $\mathbb{I}_{[a,b]}$.
We write $\lVert\cdot\rVert_2$ for the $2$-norm of a vector, or the induced $2$-norm if the argument is a matrix.
Moreover, for some positive definite $P=P^\top\succ0$, we define $\lVert x\rVert_P^2\coloneqq x^\top Px$.
The point-to-set distance w.r.t. a set $\mathcal{Z}$ is defined as $\lVert x\rVert_{\mathcal{Z}}\coloneqq\inf_{x'\in\mathcal{Z}}\lVert x'-x\rVert_2$.
We write $\lambda_{\min}(P)$ ($\lambda_{\max}(P)$) for the minimum (maximum) eigenvalue of $P$, and we define $\lambda_{\min}(P_1,P_2)\coloneqq\min\{\lambda_{\min}(P_1),\lambda_{\min}(P_2)\}$ for two symmetric matrices $P_1,P_2$, and similarly for $\lambda_{\max}(P_1,P_2)$.
Finally, $\sigma_{\min}(A)$ denotes the minimum singular value of a matrix $A$.
Throughout this paper, we use the inequalities
\begin{align}\label{eq:ab_ineq}
\lVert a+b\rVert_P^2&\leq2\lVert a\rVert_P^2+2\lVert b\rVert_P^2,\\
\label{eq:prop_proof_std_norm_ineq}
\lVert a\rVert_P^2-\lVert b\rVert_P^2&\leq\lVert a-b\rVert_P^2+2\lVert a-b\rVert_P\lVert b\rVert_P,
\end{align}
which hold for any vectors $a$, $b$ and matrix $P=P^\top\succ0$.

\section{Problem setup}\label{sec:setting}
In this paper, we consider discrete-time nonlinear systems of the form
\begin{align}\label{eq:nl_sys}
x_{t+1}&=f(x_t,u_t)=f_0(x_t)+Bu_t
\end{align}
with $f_0:\mathbb{R}^n\to\mathbb{R}^n$, $B\in\mathbb{R}^{n\times m}$, and output
\begin{align}\label{eq:sys_output}
y_t=h(x_t,u_t)=h_0(x_t)+Du_t
\end{align}
for some $h_0:\mathbb{R}^{n}\to\mathbb{R}^p$, $D\in\mathbb{R}^{p\times m}$.
We impose pointwise-in-time constraints on the input, i.e., $u_t\in\mathbb{U}$ for all $t\in\mathbb{I}_{\geq0}$, where the constraint set $\mathbb{U}$ is assumed to be compact.
Due to the inexact prediction model, including state constraints into our framework would necessitate additional robust constraint tightening methods, which is an interesting issue for future research.

We propose a state-feedback MPC scheme to track a desired setpoint reference $y^\rmr\in\mathbb{R}^p$  with the nonlinear system~\eqref{eq:nl_sys}--\eqref{eq:sys_output}.
In contrast to existing works with this goal such as~\cite{limon2018nonlinear}, the prediction model relies on the dynamics linearized at the current state $x_t$. 
In the remainder of this section, we state key assumptions that are required to prove that the MPC scheme asymptotically tracks the desired output setpoint.

\subsection{Smoothness assumptions}
We assume that both $f$ and $h$ are continuously differentiable and hence, for some $\tilde{x}\in\mathbb{R}^n$, we can define
\begin{align}\label{eq:Lin}
\begin{split}
A_{\tilde{x}}&\coloneqq\frac{\partial f_0}{\partial x}\Big\rvert_{\tilde{x}},\>\>
e_{\tilde{x}}\coloneqq f_0(\tilde{x})-A_{\tilde{x}}\tilde{x},\\
C_{\tilde{x}}&\coloneqq\frac{\partial h_0}{\partial x}\Big\rvert_{\tilde{x}},\>\>
r_{\tilde{x}}\coloneqq h_0(\tilde{x})-C_{\tilde{x}}\tilde{x}.
\end{split}
\end{align}
We write $f_{\tilde{x}}(x,u)\coloneqq A_{\tilde{x}}x+Bu+e_{\tilde{x}}$ for the system dynamics linearized at $(x,u)=(\tilde{x},0)$, and $h_{\tilde{x}}(x,u)\coloneqq C_{\tilde{x}}x+Du+r_{\tilde{x}}$ for the output linearized at $(x,u)=(\tilde{x},0)$.
Note that, since System~\eqref{eq:nl_sys} is control-affine, linearizing the dynamics at time $t$ only requires the state $x_t$ and no knowledge of the input $u_t$, which is a crucial fact for the proposed MPC scheme and its theoretical analysis.
If $f$ or $h$ are not affine in $u$, then it can be readily enforced by defining a new, incremental input $\Delta u_k\coloneqq u_{k+1}-au_k$ for some $a\in\mathbb{R}$ with $|a|\leq 1$ (e.g., $a=1$ corresponds to a standard incremental input).
In this case, the exact ``prediction model'' $u_{k+1}=au_k+\Delta u_k$ still allows us to enforce hard constraints on $u$ as well as $\Delta u$ while ensuring that the system is control-affine.
The key idea of this paper relies on the fact that a nonlinear system can be approximated locally by its linearization, given that $f$ is suitably smooth.
Note that, by definition, it holds that
\begin{align}\label{eq:f_equals_f_lin}
f_x(x,u)=f(x,u)\>\>\text{and}\>\>h_x(x,u)=h(x,u).
\end{align}
That is, locally at the linearization point, the linearization is equal to the nonlinear function $f$.
This insight is important for the theoretical analysis provided in the remainder of this paper since it implies that the prediction error of the proposed MPC scheme, which uses a local linearization-based prediction model, is zero in the first time step.
Throughout this paper, we assume that all vector fields involved in the system dynamics are twice continuously differentiable.

\begin{assumption}\label{ass:C2}
(Smoothness)
The vector fields $f_0$ and $h_0$ are twice continuously differentiable.
\end{assumption}
Assumption~\ref{ass:C2} implies a useful quantitative bound on the difference between the nonlinear vector fields $f(x,u),h(x,u)$ and their linearizations $f_{\tilde{x}}(x,u),h_{\tilde{x}}(x,u)$ for $x\neq\tilde{x}$.
\begin{proposition}\label{prop:Lin}
If Assumption~\ref{ass:C2} holds, then for any compact set $X\subset\mathbb{R}^n$, there exist $c_X,c_{Xh} >0$ such that for any $x,\tilde{x}\in X,u\in\mathbb{U}$, it holds that
\begin{align}\label{eq:ass_Lin}
\lVert f(x,u)-f_{\tilde{x}}(x,u)\rVert_2&\leq c_X\lVert x-\tilde{x}\rVert_2^2,\\\label{eq:ass_Lin_output}
\lVert h(x,u)-h_{\tilde{x}}(x,u)\rVert_2&\leq c_{Xh}\lVert x-\tilde{x}\rVert_2^2.
\end{align}
\end{proposition}
\begin{proof}
This follows directly from Taylor's Theorem in the multivariable case, together with compactness of $X$ and $\mathbb{U}$.
\end{proof}
For Proposition~\ref{prop:Lin}, it is crucial to consider compact sets $X,\mathbb{U}$, since the error bound is in general not necessarily satisfied globally.
We later show that a certain (compact) Lyapunov function sublevel set is invariant and can hence be used to define $X$.
Although using uniform constants $c_X,c_{Xh}$ in~\eqref{eq:ass_Lin} and~\eqref{eq:ass_Lin_output} over the set $X$ may be conservative, the error bound still becomes arbitrarily small if $x$ and $\tilde{x}$ are sufficiently close to each other.
Furthermore, Assumption~\ref{ass:C2} implies that $f_0$ and $h_0$ are locally Lipschitz continuous in $x$, i.e., for any $c_{\mathrm{Lip}}>0$ and any compact set $X$ there exists a constant $L_{f}\geq0$ such that for any $x,\tilde{x}\in X,u\in\mathbb{U}$ satisfying $\lVert x-\tilde{x}\rVert_2\leq c_{\text{Lip}}$ it holds that
\begin{align}\label{eq:ass_Lip}
\lVert f(x,u)-f(\tilde{x},u)\rVert_2\leq L_{f}\lVert x-\tilde{x}\rVert_2,
\end{align}
and similarly for $h$.
Using Assumption~\ref{ass:C2}, we can derive the following bound between two linear models obtained by linearizations at two different points $x,\tilde{x}$ when evaluated at two further points $x_{\rma},x_{\rmb}$:
For any constant $c_{\mathrm{Lip}}>0$ and any compact set $X$, there exist $L_f\geq0$ and $c_X>0$ such that
\begin{align}\label{eq:diff_lin}
&\lVert f_{x}(x_{\rma},u)-f_{\tilde{x}}(x_{\rmb},u)\rVert_2\\\nonumber
\leq&\lVert f(x_{\rma},u)-f_{x}(x_{\rma},u)\rVert_2+\lVert f(x_{\rmb},u)-f_{\tilde{x}}(x_{\rmb},u)\rVert_2\\\nonumber
&+\lVert f(x_{\rma},u)-f(x_{\rmb},u)\rVert_2\\\nonumber
\stackrel{\eqref{eq:ass_Lin},\eqref{eq:ass_Lip}}{\leq}&c_{X}\lVert x_{\rma}-x\rVert_2^2+c_{X}\lVert x_{\rmb}-\tilde{x}\rVert_2^2+L_f\lVert x_{\rma}-x_{\rmb}\rVert_2,
\end{align}
for all $x,x_{\rma},\tilde{x},x_{\rmb}\in X,u\in\mathbb{U}$ with $\lVert x_{\rma}-x_{\rmb}\rVert_2\leq c_{\mathrm{Lip}}$.

\subsection{Steady-state manifold}
The control goal is to steer the nonlinear system~\eqref{eq:nl_sys} to a desired target setpoint, i.e., to track a user-specified output $y^\rmr\in\mathbb{R}^p$.
Since the output $y$ may depend on the input $u$, this also allows us to consider input setpoints.
Let us now define the set of all feasible steady-states
\begin{align}
\mathcal{Z}^\rms\coloneqq\{(x^\rms,u^\rms)\in\mathbb{R}^{n}\times\mathbb{U}^\rms\mid x^\rms=f(x^\rms,u^\rms)\}
\end{align}
with some (user-chosen) convex and compact set $\mathbb{U}^\rms\subseteq\text{int}\left(\mathbb{U}\right)$, which is required for a local controllability argument in our proofs.
Further, we denote the projection of $\mathcal{Z}^\rms$ on the state component by $\mathcal{Z}^\rms_x$ and the projection of $\mathcal{Z}^\rms$ on the output as
\begin{align*}
\mathcal{Z}_{\rmy}^\rms\coloneqq\{y^\rms\in\mathbb{R}^p\mid \exists(x^\rms,u^\rms)\in\mathcal{Z}^\rms:y^\rms=h(x^\rms,u^\rms)\}.
\end{align*}
The optimal equilibrium cost is defined as
\begin{align}\label{eq:opt_equil_prob}
J_\rmeq^*\coloneqq\min_{y^\rms\in\mathcal{Z}_{\rmy}^\rms}&\lVert y^\rms-y^\rmr\rVert_S^2
\end{align}
for some $S\succ0$.
We denote a minimizer of~\eqref{eq:opt_equil_prob} by $y^{\rms\rmr}$.
In Section~\ref{sec:MPC_stab}, we provide sufficient conditions under which this minimizer is unique. 
Let us define the set of equilibria of the linearized system at some state $\tilde{x}$ as
\begin{align*}
\mathcal{Z}^\rms_\rmlin(\tilde{x})\coloneqq\big\{(x^\rms,u^\rms)\in\mathbb{R}^n\times\mathbb{U}^\rms\mid x^\rms=A_{\tilde{x}}x^\rms+Bu^\rms+e_{\tilde{x}}\}
\end{align*}
and the projection on the output as
\begin{align*}
\mathcal{Z}_{\rmy,\rmlin}^\rms(\tilde{x})\coloneqq\big\{y^\rms\in\mathbb{R}^p\mid\exists(&x^\rms,u^\rms)\in\mathcal{Z}^\rms_\rmlin(\tilde{x}):\\
&y^\rms=C_{\tilde{x}}x^\rms+Du^\rms+r_{\tilde{x}}\big\}.
\end{align*}
The optimal reachable equilibrium of the linearized system at $\tilde{x}$ is the minimizer of
\begin{align}\label{eq:opt_equil_prob_lin}
J_{\rmeq,\rmlin}^*(\tilde{x})\coloneqq\min_{y^\rms\in\mathcal{Z}_{\rmy,\rmlin}^\rms(\tilde{x})}\>\>&\lVert y^\rms-y^\rmr\rVert_S^2.
\end{align}
We denote the minimizer of~\eqref{eq:opt_equil_prob_lin} by $y^{\rms\rmr}_{\rmlin}(\tilde{x})$.
It follows from $S\succ0$ that Problem~\eqref{eq:opt_equil_prob_lin} is strongly convex at any linearization point, i.e., for any $\tilde{x}\in\mathbb{R}^n$, $y^\rms\in\mathcal{Z}_{\rmy,\rmlin}^\rms(\tilde{x})$, it holds that
\begin{align}\label{eq:ass_strongly_convex}
\lVert y^\rms-y^\rmr\rVert_S^2-J_{\mathrm{eq},\rmlin}^*(\tilde{x})\geq \lVert y^\rms-y^{\rms\rmr}_{\rmlin}(\tilde{x})\rVert_S^2,
\end{align}
compare~\cite[Inequality (11)]{koehler2020nonlinear}.
\begin{assumption}\label{ass:unique_steady_state}
(Unique steady-state)
There exists $\sigma_\rms>0$ such that, for any $\tilde{x}\in\mathbb{R}^n$,
\begin{align}\label{eq:ass_unique_steady_state}
\sigma_{\min}\left(\begin{bmatrix}A_{\tilde{x}}-I&B\\C_{\tilde{x}}&D\end{bmatrix}\right)\geq\sigma_\rms.
\end{align}
\end{assumption}
Assumption~\ref{ass:unique_steady_state} implies that $\begin{bmatrix}A_{\tilde{x}}-I&B\\C_{\tilde{x}}&D\end{bmatrix}$ has full column rank, which is a standard condition in tracking, compare~\cite[Lemma~1.8]{rawlings2020model}, \cite[Remark~1]{limon2018nonlinear}.
This condition means that for any steady-state output of the linearized system, the corresponding input-state pair is unique.
More precisely, for any $\tilde{x}\in \mathbb{R}^n$, there exists a linear map $\hat{g}_{\tilde{x}}:\mathcal{Z}_{\rmy,\rmlin}^\rms(\tilde{x})\to\mathcal{Z}^\rms_\rmlin(\tilde{x})$ such that
\begin{align}\label{eq:ass_unique_steady_state_maps}
\hat{g}_{\tilde{x}}(y^\rms)=(x^\rms,u^\rms),
\end{align}
where $(x^\rms,u^\rms)\in\mathcal{Z}^\rms_\rmlin(\tilde{x})$ is the steady-state corresponding to $y^\rms$, i.e., $y^\rms=h_{\tilde{x}}(x^\rms,u^\rms)$.
Due to the uniform lower bound~\eqref{eq:ass_unique_steady_state}, the map $\hat{g}_{\tilde{x}}$ is uniformly Lipschitz continuous for all $\tilde{x}\in\mathbb{R}^n$.
Further, the condition~\eqref{eq:ass_unique_steady_state} implies that the linearized dynamics have no transmission zeros at~$1$~\cite[Ass.~1]{magni2005solution} and that the number of outputs $p$ is greater than or equal to the number of inputs $m$.
Finally, by a global version of the inverse function theorem~\cite[Condition (1.1)]{radulescu1980global}, Assumption~\ref{ass:unique_steady_state} implies the existence of 
a unique equilibrium $(x^\rms,u^\rms)\in\mathcal{Z}^\rms$ for the \emph{nonlinear} system for any given output equilibrium $y^\rms\in\mathcal{Z}_{\rmy}^\rms$, compare also~\cite[Remark~1]{limon2018nonlinear}.
Throughout the paper, we write
\begin{align}
(x^{\rms\rmr}_{\rmlin}(\tilde{x}),u^{\rms\rmr}_{\rmlin}(\tilde{x}))\coloneqq \hat{g}_{\tilde{x}}(y^{\rms\rmr}_{\rmlin}(\tilde{x}))
\end{align}
for the unique input-state pair corresponding to the minimizer $y^{\rms\rmr}_{\rmlin}(\tilde{x})$ of~\eqref{eq:opt_equil_prob_lin}, and we write $(x^{\rms\rmr},u^{\rms\rmr})$ for the input-state pair corresponding to $y^{\rms\rmr}$, i.e.,
\begin{align}\label{eq:x_sr_u_sr_def}
(x^{\rms\rmr},u^{\rms\rmr})=\hat{g}_{x^{\rms\rmr}}(y^{\rms\rmr}).
\end{align}
Moreover, we require that the linearized dynamics are uniformly controllable.
\begin{assumption}\label{ass:ctrb}
(Controllability)
The pair $(A_{\tilde{x}},B)$ is uniformly controllable for all $\tilde{x}\in\mathbb{R}^n$, i.e., the minimum singular value of $\begin{bmatrix}B&\dots&A^{n-1}_{\tilde{x}}B\end{bmatrix}$ is uniformly lower bounded.
\end{assumption}
Using standard arguments from linear systems theory, compare~\cite[Theorem 5]{sontag1998mathematical}, it can be shown that Assumption~\ref{ass:ctrb} 
implies the existence of a constant $\Gamma>0$ such that for any $\tilde{x}\in\mathbb{R}^n,(x^\rms,u^\rms)\in\mathcal{Z}^\rms_\rmlin(\tilde{x})$, and any initial condition $x_0$, there exists an input trajectory $\hat{u}_k\in\mathbb{R}^m$, $k\in\mathbb{I}_{[0,n-1]}$,
steering the linearized system from $x_0$ to $x^\rms$, i.e, $\hat{x}_0=x_0,\hat{x}_{k+1}=f_{\tilde{x}}(\hat{x}_k,\hat{u}_k),\hat{x}_n=x^\rms$, while satisfying
\begin{align}\label{eq:ass_ctrb}
\sum_{k=0}^{n-1}\left\lVert \hat{x}_k-x^\rms\right\rVert_2+\lVert \hat{u}_k-u^\rms\rVert_2
\leq \Gamma\lVert x^\rms-x_0\rVert_2.
\end{align}
Moreover, the following assumption is required for the linearized system dynamics at any state.

\begin{assumption}\label{ass:exist_equil_state}
(Non-singular dynamics)
There exists $\underline{\sigma}>0$ such that $\sigma_{\min}(I-A_{\tilde{x}})\geq\underline{\sigma}$ for any $\tilde{x}\in \mathbb{R}^n$.
\end{assumption}

Assumption~\ref{ass:exist_equil_state} implies that, for any $\tilde{x}\in \mathbb{R}^n$, the matrix $I-A_{\tilde{x}}$ has full rank and hence, for any equilibrium input $u^\rms$ there exists a unique equilibrium state $x^\rms$ such that
\begin{align*}
x^\rms=A_{\tilde{x}}x^\rms+Bu^\rms+e_{\tilde{x}}.
\end{align*}
%
In case that $\mathbb{U}=\mathbb{R}^m$, it is straightforward to relax Assumption~\ref{ass:exist_equil_state} by requiring that there exists a state-feedback gain $K$ such that $A_K=A_{\tilde{x}}+BK$ satisfies the non-singularity condition $\sigma_{\min}(I-A_K)\geq\underline{\sigma}$ for some $\underline{\sigma}$ and for all $\tilde{x}\in\mathbb{R}^n$.
We conjecture that it is possible to relax Assumption~\ref{ass:exist_equil_state} further at the price of a more involved analysis.

\begin{remark}
Note that the conditions in Assumptions~\ref{ass:unique_steady_state}--\ref{ass:exist_equil_state} are imposed for all $\tilde{x}\in\bbr^n$.
This is mainly done for notational convenience.
As will become clear in our theoretical results, it actually suffices if these assumptions hold for all $\tilde{x}$ in a suitably defined compact set depending on the (positively invariant) sublevel set of the Lyapunov function used to prove stability.
\end{remark}

Finally, we make an additional assumption on the steady-state manifold of the linearized dynamics.
\begin{assumption}\label{ass:steady_state_compact}
(Compact steady-state manifold)
There exists a compact set $\mathcal{B}$ such that $\mathcal{Z}^\rms_\rmlin(\tilde{x})\subseteq\mathcal{B}$ for all $\tilde{x}\in\mathbb{R}^n$.
\end{assumption}
Assumption~\ref{ass:steady_state_compact} means that the union of all steady-state manifolds for the linearized dynamics at any point is contained in a compact set.
If the input equilibrium constraints $\mathbb{U}^\rms$ are compact and Assumption~\ref{ass:exist_equil_state} holds, then this is satisfied if 
the Jacobian is uniformly bounded.
Assumption~\ref{ass:steady_state_compact} is required for our theoretical results to obtain a uniform bound on the optimal equilibrium cost~\eqref{eq:opt_equil_prob_lin} and to conclude compactness of certain Lyapunov function sublevel sets.
The assumption can be dropped if it is known that the closed-loop trajectories lie within a compact invariant subset of the state-space.

\section{Linear tracking MPC for nonlinear systems}\label{sec:MPC}
In this section, we propose a linear tracking MPC scheme to steer the nonlinear system~\eqref{eq:nl_sys}--\eqref{eq:sys_output} to a desired target setpoint.
The key idea is to use a local linearization-based model of the nonlinear system for prediction, and to update the linearization online using the current measurements.
After stating the scheme in Section~\ref{sec:MPC_scheme}, we prove lower and upper bounds on the optimal value function of the MPC problem in Section~\ref{sec:MPC_value}.
Further, a useful contraction property of the Lyapunov function is stated in Section~\ref{sec:MPC_feas2}.
Section~\ref{sec:MPC_stab} contains the main result on closed-loop exponential stability of the optimal reachable equilibrium.

\subsection{MPC scheme}\label{sec:MPC_scheme}
Given the current state $x_t$ at time $t$ as well as the linearization of the nonlinear system at $x_t$ according to~\eqref{eq:Lin}, the following optimal control problem will be the basis for our proposed MPC scheme:
\begin{subequations}\label{eq:MPC}
\begin{align}
\underset{\substack{\bar{x}(t),\bar{u}(t)\\x^\rms(t),u^\rms(t)\\y^\rms(t)}}{\min}\>\>&\sum_{k=0}^{N-1} \lVert \bar{x}_k(t)-x^\rms(t)\rVert_Q^2+\lVert\bar{u}_k(t)-u^\rms(t)\rVert_R^2\\\nonumber
&\,\>\,+\lVert y^\rms(t)-y^\rmr\rVert_S^2\\\label{eq:MPC1}
\text{s.t.}\quad&\bar{x}_{k+1}(t)=A_{x_t}\bar{x}_k(t)+B\bar{u}_k(t)+e_{x_t},\\\label{eq:MPC2}
&\bar{x}_0(t)=x_t,\>\>\bar{x}_N(t)=x^\rms(t),\\\label{eq:MPC3}
&\bar{u}_k(t)\in\mathbb{U},\>\>k\in\mathbb{I}_{[0,N-1]},\\
\label{eq:MPC4}
&(x^\rms(t),u^\rms(t))\in\mathcal{Z}^\rms_\rmlin(x_t),\\\label{eq:MPC5}
&y^\rms(t)=C_{x_t}x^\rms(t)+Du^\rms(t)+r_{x_t}.
\end{align}
\end{subequations}
Here, $\bar{x}(t)\in\mathbb{R}^{nN}$ and $\bar{u}(t)\in\mathbb{R}^{mN}$ denote the predicted state and input trajectory at time $t$, taking the value $\bar{x}_k(t)\in\mathbb{R}^n$ and $\bar{u}_k(t)\in\mathbb{R}^m$ at the $k$-th (predicted) time step, respectively.
Compared to a standard linear MPC scheme with terminal equality constraints (compare~\cite{rawlings2020model}), Problem~\eqref{eq:MPC} has two additional ingredients.
First, the present scheme contains an artificial setpoint $(x^\rms(t),u^\rms(t))$ which is optimized online and whose distance w.r.t. the desired target setpoint $y^\rmr$ is penalized in the cost, similar to the tracking MPC formulation in~\cite{limon2018nonlinear}.
If compared to classical MPC schemes with terminal equality constraints~\cite{rawlings2020model}, such an artificial steady-state increases the region of attraction and leads to recursive feasibility despite online setpoint changes~\cite{limon2018nonlinear}.
Further, the prediction of future trajectories of the present nonlinear system is not based on the full nonlinear model~\eqref{eq:nl_sys}, but instead on a local linearization around the current state $x_t$.
Therefore, also the artificial setpoint $x^\rms(t)$ is an equilibrium for the linearized dynamics according to~\eqref{eq:MPC4}, but not necessarily for the nonlinear system, and the artificial output setpoint $y^\rms(t)$ satisfies the linearized output equation.

We assume that the cost matrices in~\eqref{eq:MPC} are positive definite, i.e., $Q,R\succ0$.
Under the given assumptions, it can be shown analogously to~\cite[Proposition 1]{berberich2021linearpart2} that the optimal solution of~\eqref{eq:MPC} is unique.
We denote this optimal solution by $\bar{x}^*(t),\bar{u}^*(t),x^{\rms*}(t),u^{\rms*}(t),y^{\rms*}(t)$ and the corresponding optimal cost by $J_N^*(x_t)$.
Our results can be extended to positive semidefinite state weightings $Q\succeq0$ by invoking an input-output-to-state stability argument (compare~\cite{berberich2020tracking,cai2008input}).
If $\mathbb{U}$, $\mathbb{U}^\rms$ are polytopic, then Problem~\eqref{eq:MPC} is a convex QP and can be solved efficiently.
On the contrary, solving the non-convex problems associated with nonlinear MPC to optimality is in general computationally intractable.
It is worth noting that the computational complexity of the proposed MPC approach is also smaller than that of alternative linearization-based approaches such as the real-time iteration scheme~\cite{diehl2005nominal,liao2020time,zanelli2021lyapunov} since i) only the linearization w.r.t.\ $x_t$ instead of the previously optimal solution $\bar{x}^*(t-1)$ is needed and ii) comparable stability guarantees of the real-time iteration scheme require a sufficiently small sampling time~\cite{zanelli2021lyapunov}, i.e., solving the underlying optimization problem more frequently~\cite{liao2020time}, cf. Remark~\ref{rk:RTI}.

\begin{algorithm}
\begin{Algorithm}\label{alg:MPC}
\normalfont{\textbf{$n$-step tracking MPC Scheme}}
\begin{enumerate}
\item At time $t$, take the current state measurement $x_t$ and compute $A_{x_t},e_{x_t},C_{x_t},D,r_{x_t}$ according to~\eqref{eq:Lin}.
\item Solve~\eqref{eq:MPC} and apply the first $n$ input components $u_{t+k}=\bar{u}_{k}^*(t)$, $k\in\mathbb{I}_{[0,n-1]}$.
\item Set $t=t+n$ and go back to 1).
\end{enumerate}
\end{Algorithm}
\end{algorithm}

In this paper, we 
consider the $n$-step MPC scheme outlined in Algorithm~\ref{alg:MPC}.
We employ a multi-step MPC scheme (compare~\cite{gruene2015robustness,worthmann2017interaction}) due to the joint occurrence of a model mismatch (induced by the linearized model) and terminal equality constraints.
More precisely, the candidate solution used in our theoretical analysis is defined based on the shifted previously optimal solution with an appended deadbeat controller, requiring $n$ additional time steps, i.e., a multi-step MPC scheme with at least $n$ steps, cf.~\cite[Figure 1]{berberich2021guarantees}.
The same theoretical guarantees can be given for a $\nu$-step MPC scheme, where $\nu\leq n$ denotes the controllability index, provided that Assumption~\ref{ass:ctrb} is modified accordingly.

Clearly, the prediction model in the proposed MPC scheme is not exact due to the linearization of the nonlinear system.
As we will see in the remainder of this paper, by suitably tuning the cost parameters ($S$ needs to be small) and when starting close to the steady-state manifold $\mathcal{Z}^\rms$ of the nonlinear system, the artificial steady-state $x^\rms(t)$ is always close to the current state $x_t$ such that the prediction error is sufficiently small.
Then, $x^\rms(t)$ remains close to the nonlinear steady-state manifold and slowly drifts towards the optimal reachable equilibrium $x^{\rms\rmr}$ such that, asymptotically, the closed-loop state trajectory converges to $x^{\rms\rmr}$, compare Figure~\ref{fig:manifold}.

\subsection{Value function bound}\label{sec:MPC_value}

In order to prove closed-loop exponential stability, we consider a Lyapunov function candidate of the form 
\begin{align*}
V(x_t)\coloneqq J_N^*(x_t)-J_{\mathrm{eq},\rmlin}^*(x_t)
\end{align*}
with $J_{\mathrm{eq},\rmlin}^*(x_t)$ as in~\eqref{eq:opt_equil_prob_lin}.
The following result shows that $V$ admits suitable quadratic lower and upper bounds, which will be required to prove desired stability properties.

\begin{lemma}\label{lem:value_fcn_bound}
Suppose $N\geq n$ and Assumptions~\ref{ass:C2},~\ref{ass:unique_steady_state}, and~\ref{ass:ctrb} hold.
For any compact set $X\subset\mathbb{R}^n$, there exist $c_\rml,c_\rmu,\delta>0$ such that
\begin{itemize}
\item[(i)] for all $x_t\in X$, the function $V$ is lower bounded as
\begin{align}\label{eq:lem_value_fcn_lower_bound}
V(x_t)\geq c_\rml\lVert x_t-x^{\rms\rmr}_{\rmlin}(x_t)\rVert_2^2,
\end{align}
\item[(ii)] for all $x_t\in X$ with $\lVert x_t-x^{\rms\rmr}_{\rmlin}(x_t)\rVert_2\leq\delta$, the function $V$ is upper bounded as
\begin{align}\label{eq:lem_value_fcn_upper_bound}
V(x_t)\leq c_\rmu\lVert x_t-x^{\rms\rmr}_{\rmlin}(x_t)\rVert_2^2.
\end{align}
\end{itemize}
\end{lemma}
\begin{proof}
\textbf{(i) Lower bound}\\
Note that
\begin{align*}
&V(x_t)=J_N^*(x_t)-J_{\mathrm{eq},\rmlin}^*(x_t)\\
&\geq \lVert x_t-x^{\rms*}(t)\rVert_Q^2+\lVert y^{\rms*}(t)-y^\rmr\rVert_S^2-J_{\mathrm{eq},\rmlin}^*(x_t)\\
&\stackrel{\eqref{eq:ass_strongly_convex}}{\geq}\lVert x_t-x^{\rms*}(t)\rVert_Q^2+\lambda_{\min}(S)\lVert y^{\rms*}(t)-y^{\rms\rmr}_{\rmlin}(x_t)\rVert_2^2.
\end{align*}
Assumption~\ref{ass:unique_steady_state} (i.e., the maps in~\eqref{eq:ass_unique_steady_state_maps}) implies the existence of a constant $\hat{c}_\rml>0$, which can be chosen uniformly over $x_t$, such that 
\begin{align*}
\lVert y^{\rms*}(t)-y^{\rms\rmr}_{\rmlin}(x_t)\rVert_2^2\geq \hat{c}_\rml\lVert x^{\rms*}(t)-x^{\rms\rmr}_{\rmlin}(x_t)\rVert_2^2.
\end{align*}
Hence, we obtain
\begin{align*}
V(x_t)&\geq \lVert x_t-x^{\rms*}(t)\rVert_Q^2+\hat{c}_\rml\lambda_{\min}(S)\lVert x^{\rms*}(t)-x^{\rms\rmr}_{\rmlin}(x_t)\rVert_2^2\\
&\stackrel{\eqref{eq:ab_ineq}}{\geq}\underbrace{\frac{\min\{\lambda_{\min}(Q),\hat{c}_\rml\lambda_{\min}(S)\}}{2}}_{c_\rml\coloneqq}\lVert x_t-x^{\rms\rmr}_{\rmlin}(x_t)\rVert_2^2.
\end{align*}
\\
\textbf{(ii) Upper bound}\\
In the following, we construct a candidate solution to Problem~\eqref{eq:MPC} which will then be used to bound the optimal cost $J_N^*(x_t)$.
We choose the candidate equilibrium as $x^\rms(t)=x^{\rms\rmr}_{\rmlin}(x_t),u^\rms(t)=u^{\rms\rmr}_{\rmlin}(x_t),y^\rms(t)=y^{\rms\rmr}_{\rmlin}(x_t)$, i.e., the optimal reachable equilibrium of the system linearized at $x_t$.
Using $N\geq n$ and Assumption~\ref{ass:ctrb}, there exists a 
trajectory of the linearized dynamics steering the state to $x^\rms(t)$ within $N$ steps while satisfying
\begin{align*}
\sum_{k=0}^{N-1}\lVert \bar{x}_k(t)-x^\rms(t)\rVert_2+\lVert\bar{u}_k(t)-u^\rms(t)\rVert_2\leq\Gamma\lVert x_t-x^\rms(t)\rVert_2
\end{align*}
for some $\Gamma>0$ (compare~\eqref{eq:ass_ctrb}).
If $\delta$ is sufficiently small, then $u^{\rms\rmr}_{\rmlin}(x_t)\in\text{int}(\mathbb{U})$ implies that the corresponding input satisfies the constraints, i.e., $\bar{u}_k(t)\in\mathbb{U}$ for all $k\in\mathbb{I}_{[0,N-1]}$.
Clearly, the above inequality implies
\begin{align*}
&\quad\sum_{k=0}^{N-1}\lVert \bar{x}_k(t)-x^\rms(t)\rVert_Q^2+\lVert\bar{u}_k(t)-u^\rms(t)\rVert_R^2
\\
&\leq\lambda_{\max}(Q,R)\sum_{k=0}^{N-1}(\lVert \bar{x}_k(t)-x^\rms(t)\rVert_2^2+\lVert\bar{u}_k(t)-u^\rms(t)\rVert_2^2)\\
&\leq\lambda_{\max}(Q,R)\Gamma^2\lVert x_t-x^\rms(t)\rVert_2^2.
\end{align*}
Thus, the following holds for the Lyapunov function candidate
\begin{align*}
V(x_t)\leq &\lambda_{\max}(Q,R)\Gamma^2\lVert x_t-x^\rms(t)\rVert_2^2+\lVert y^\rms(t)-y^\rmr\rVert_S^2\\
&-J_{\mathrm{eq},\rmlin}^*(x_t)\\
=&\underbrace{\lambda_{\max}(Q,R)\Gamma^2}_{c_\rmu\coloneqq}\lVert x_t-x^{\rms\rmr}_{\rmlin}(x_t)\rVert_2^2.\qedhere
\end{align*}
\end{proof}

Lemma~\ref{lem:value_fcn_bound} provides bounds on the Lyapunov function candidate $V(x_t)$ that will be employed to prove closed-loop exponential stability.
As an alternative to $V(x_t)$, one could consider the candidate $J_N^*(x_t)-J_{\mathrm{eq}}^*$, which depends on the cost of the optimal reachable equilibrium for the \emph{nonlinear} system instead of the cost $J_{\mathrm{eq},\rmlin}^*(x_t)$ for the linearized system.
However, deriving a useful lower bound for $J_N^*(x_t)-J_{\mathrm{eq}}^*$ similar to~\eqref{eq:lem_value_fcn_lower_bound} is difficult, which is why we consider the proposed Lyapunov function candidate $V(x_t)$ instead.

\subsection{Contraction property}\label{sec:MPC_feas2}

The following result shows that feasibility of Problem~\eqref{eq:MPC} at time $t$ implies, under additional assumptions, feasibility at time $t+n$ and a certain contraction property for the Lyapunov function candidate $V$ on suitable sublevel sets.

\begin{proposition}\label{prop:rec_feas2}
Suppose $N\geq n$ and Assumptions~\ref{ass:C2}-\ref{ass:steady_state_compact} hold.
Then, there exist $V_{\max},J_{\mathrm{eq}}^{\max}>0$ such that, if $V(x_t)\leq V_{\max}$, $J_{\mathrm{eq},\rmlin}^*(x_t)\leq J_{\mathrm{eq}}^{\max}$, $J_{\mathrm{eq},\rmlin}^*(x_{t+n})\leq J_{\mathrm{eq}}^{\max}$, then Problem~\eqref{eq:MPC} is feasible at time $t+n$ and there exists a constant $0<c_\rmV<1$ such that $V(x_{t+n})\leq c_\rmV V(x_t)$.
\end{proposition}

The proof of Proposition~\ref{prop:rec_feas2} is provided in the appendix, and it uses a case distinction with two different candidate solutions.
In Appendix~\ref{appendix_A}, the case that the tracking cost w.r.t. the artificial steady-state is relatively large is considered, compare Inequality~\eqref{eq:prop_case1}.
In this case, the candidate solution is the previously optimal input appended by a local deadbeat controller compensating the model mismatch due to the linearization.
If $V_{\max}$ and $J_{\mathrm{eq}}^{\max}$ are sufficiently small, then this model mismatch is also small such that a decrease of the optimal cost can be derived.
On the other hand, Appendix~\ref{appendix_B} considers the converse case, where the current state is close to the artificial steady-state.
In this case, the candidate solution results from shifting the artificial equilibrium towards the optimal reachable equilibrium for the linearized dynamics.
Proposition~\ref{prop:rec_feas2} can also be seen as an extension of~\cite{limon2018nonlinear,koehler2020nonlinear}, where similar properties are shown for MPC with a \emph{nonlinear} prediction model, whereas our MPC scheme contains a linear prediction model.

Even for $V_{\max}$ arbitrarily small, the bound $V(x_t)\leq V_{\max}$ holds in a neighborhood of the steady-state manifold if $S$ is chosen sufficiently small using compactness (cf. the proof of Theorem~\ref{thm:stab}).
In this case, $x^{\rms*}(t)$ is close to $x_t$ and hence the stage cost $\sum_{k=0}^{N-1}\lVert\bar{x}_k^*(t)-x^{\rms*}(t)\rVert_Q^2+\lVert\bar{u}_k^*(t)-u^{\rms*}(t)\rVert_R^2$ becomes small.
Similarly, as we exploit in the next section, the bounds $J_{\mathrm{eq},\rmlin}^*(x_t)\leq J_{\mathrm{eq}}^{\max}$ and $J_{\mathrm{eq},\rmlin}^*(x_{t+n})\leq J_{\mathrm{eq}}^{\max}$ hold with arbitrarily small $J_{\mathrm{eq}}^{\max}$ if $S$ is chosen sufficiently small since the steady-state manifold is compact by Assumption~\ref{ass:steady_state_compact}.
Note that Proposition~\ref{prop:rec_feas2} does not prove any recursive closed-loop properties since the assumed bounds are not necessarily satisfied recursively.
We discuss these conditions in relation with the proof of closed-loop recursive feasibility and stability in Section~\ref{sec:MPC_stab}.

\subsection{Exponential stability}\label{sec:MPC_stab}

In this section, we use Lemma~\ref{lem:value_fcn_bound} and Proposition~\ref{prop:rec_feas2} to prove closed-loop exponential stability of the optimal reachable steady-state $x^{\rms\rmr}$ of the nonlinear system.
It follows from Proposition~\ref{prop:rec_feas2} that, under the given assumptions, the function $V$ satisfies $V(x_{t+n})\leq c_\rmV V(x_t)$.
However, this does not yet prove the desired stability result since 1) it remains to show that the inequality is satisfied recursively and 2) $V$ is only lower and upper bounded by the distance w.r.t. the optimal reachable steady-state of the \emph{linearized} dynamics (compare Lemma~\ref{lem:value_fcn_bound}).
In the following, we make the additional assumption that the current state $x_t$ is close to $x^{\rms\rmr}_{\rmlin}(x_t)$ if and only if it is close to the optimal reachable steady-state $x^{\rms\rmr}$.
\begin{assumption}\label{ass:convex_steady_state_manifold}
For any compact set $X$ with $\mathcal{B}\subseteq X\times\mathbb{U}$ (cf. Assumption~\ref{ass:steady_state_compact}), there exist constants $c_{\mathrm{eq},1},c_{\mathrm{eq},2}>0$ such that, for any $\hat{x}\in X$, it holds that
\begin{align}\label{eq:convexity_nonlinear}
c_{\mathrm{eq},1}\lVert\hat{x}-x^{\rms\rmr}_{\rmlin}(\hat{x})\rVert_2^2\leq\lVert\hat{x}-x^{\rms\rmr}\rVert_2^2\leq c_{\mathrm{eq},2}\lVert\hat{x}-x^{\rms\rmr}_{\rmlin}(\hat{x})\rVert_2^2.
\end{align}
\end{assumption}
Assumption~\ref{ass:convex_steady_state_manifold} requires that the distance between some state $\hat{x}$ and $x^{\rms\rmr}$ is lower and upper bounded by the distance between $\hat{x}$ and the optimal reachable steady-state for the linearized dynamics $x^{\rms\rmr}_{\rmlin}(\hat{x})$.
In Appendix~\ref{app:cond_ass}, we show that Assumption~\ref{ass:convex_steady_state_manifold} holds if, in addition to the assumptions in Section~\ref{sec:setting}, the setpoint $y^\rmr$ is reachable and $m=p$ holds.
The following theorem shows that the optimal reachable equilibrium of the nonlinear system is exponentially stable under the proposed MPC scheme.

\begin{theorem}\label{thm:stab}
Suppose $N\geq n$ and Assumptions~\ref{ass:C2}--\ref{ass:convex_steady_state_manifold} hold.
Then, there exist $V_{\max},\bar{S}>0$ such that, if $\lambda_{\max}(S)\leq\bar{S}$ and $V(x_0)\leq V_{\max}$, then Problem~\eqref{eq:MPC} is feasible at any time $n\cdot i,i\in\mathbb{I}_{\geq0}$ and $x^{\rms\rmr}$ in~\eqref{eq:x_sr_u_sr_def} is exponentially stable, i.e., there exist constants $C>0,c_\rmV<1$ such that for all $i\in\mathbb{I}_{\geq0}$
\begin{align}\label{eq:thm_stab}
\lVert x_{ni}-x^{\rms\rmr}\rVert_2^2\leq c_\rmV^{i}C\lVert x_0-x^{\rms\rmr}\rVert_2^2.
\end{align}
\end{theorem}
\begin{proof}
Assumption~\ref{ass:steady_state_compact} implies that the union of all output equilibrium manifolds $\mathcal{Z}^\rms_{\rmy,\rmlin}(\tilde{x})$ is compact.
Thus, there exists a uniform upper bound $J_{\mathrm{eq}}^{\max}$ on $J_{\mathrm{eq},\rmlin}^*(\tilde{x})$, i.e.,
\begin{align}\label{eq:thm_stab_J_eq_max}
J_{\mathrm{eq},\rmlin}^*(\tilde{x})\leq J_{\mathrm{eq}}^{\max}\>\>\text{for all}\>\>\tilde{x}\in\mathbb{R}^n.
\end{align}
Note that $J_{\mathrm{eq}}^{\max}$ can be chosen arbitrarily small when $\lambda_{\max}(S)$ is sufficiently small.
Hence, choosing $\lambda_{\max}(S)$ sufficiently small and using~\eqref{eq:thm_stab_J_eq_max}, we can apply Proposition~\ref{prop:rec_feas2} to conclude $V(x_{t+n})\leq c_\rmV V(x_t)$, which in turn implies $V(x_{t+n})\leq V_{\max}$.
Applying this argument inductively, we conclude $V(x_{t+n})\leq c_\rmV V(x_t)$ for all $t=n\cdot i,i\in\mathbb{I}_{\geq0}$, where $c_\rmV<1$.
Using Lemma~\ref{lem:value_fcn_bound} (the upper bound holds if $V_{\max}$ is sufficiently small), this implies
\begin{align}\label{eq:thm_stab_ineq}
\lVert x_{t+n}-x^{\rms\rmr}_{\rmlin}(x_{t+n})\rVert_2^2\leq c_\rmV^i\frac{c_\rmu}{c_\rml}\lVert x_t-x^{\rms\rmr}_{\rmlin}(x_t)\rVert_2^2.
\end{align}
Finally, using~\eqref{eq:convexity_nonlinear}, this leads to~\eqref{eq:thm_stab} with $C\coloneqq\frac{c_\rmu c_{\mathrm{eq},2}}{c_\rml c_{\mathrm{eq},1}}$.
\end{proof}

Theorem~\ref{thm:stab} is our main stability result.
It shows that, if $V_{\max}$ and $S$ are sufficiently small, then the optimal reachable steady-state $x^{\rms\rmr}$ is exponentially stable under the proposed $n$-step MPC scheme, i.e., Inequality~\eqref{eq:thm_stab} holds, and hence also the output $y_t$ exponentially converges towards the optimal reachable output $y^{\rms\rmr}$.
Intuitively, Theorem~\ref{thm:stab} shows that the artificial steady-state $x^{\rms*}(t)$ and thus also the state trajectory $x_t$ slides along the steady-state manifold in closed loop towards the optimal reachable steady-state.
Thus, the guaranteed region of attraction of $x^{\rms\rmr}$ is a neighborhood around the steady-state manifold, which increases if $S$ is chosen smaller (for a given value of $V_{\max}$).
Note that, although the prediction model of the proposed MPC scheme is not exact, the closed loop is nevertheless exponentially stable (i.e., it is not only practically stable) since the prediction accuracy improves as $x_t$ gets closer to the steady-state manifold $\mathcal{Z}^\rms$.

Theorem~\ref{thm:stab} should be interpreted as a \emph{qualitative} result since it does not provide explicit values of $V_{\max}$ and $\bar{S}$ leading to closed-loop stability.
Loosely speaking, Theorem~\ref{thm:stab} guarantees stability if the initial state is \emph{sufficiently close} to $\calZ_x^\rms$ and $\lambda_{\max}(S)$ is \emph{sufficiently small}.
This is due to the fact that $V(x_t)\leq V_{\max}$ can be ensured for an arbitrarily small $V_{\max}$ if $x_t$ lies in a neighborhood of $\calZ_x^\rms$ and $\bar{S}$ is sufficiently small.
In~\cite{ferramosca2011optimal}, it is shown for linear systems that a tracking MPC formulation recovers optimality properties of a standard MPC scheme if the weight on the distance between the artificial setpoint and the reference setpoint (i.e., the matrix $S$ in our setting) is suitably large.
This indicates a trade-off when designing the matrix $S$:
It needs to be suitably small such that the linearization error is small and stability can be guaranteed, but the performance deteriorates if $S$ is chosen too small.

\begin{remark}\label{rk:RTI}
Note that Theorem~\ref{thm:stab} guarantees closed-loop stability of Algorithm~\ref{alg:MPC} which only requires solving one convex QP online.
Alternative nonlinear MPC approaches based on convex optimization typically employ an LTI prediction model by linearizing at the setpoint or, as in the real-time iteration scheme~\cite{diehl2005nominal,liao2020time,zanelli2021lyapunov}, an LTV prediction model by linearizing along the candidate solution.
Closed-loop guarantees of such approaches require either an additional bounding of the linearization error~\cite{cannon2011robust} or a sufficiently small sampling time~\cite{zanelli2021lyapunov}, i.e., solving the underlying optimization problem more frequently~\cite{liao2020time}.
The latter ensures that the linearized dynamics do not change too rapidly, which is analogous to our condition on $\lambda_{\max}(S)$ being sufficiently small.
The proposed approach has a large region of attraction due to the online optimization of the artificial steady-state.
In particular, while a standard (linearization-based) MPC only guarantees stability when starting in a region around the setpoint, our MPC scheme ensures stability for initial conditions far away from the setpoint as long as they are close to the steady-state manifold.
This fact will be illustrated with a numerical example in Section~\ref{sec:example}.
Furthermore, as is standard in MPC, our guarantees remain true if the QP is not solved up to optimality by using a warm start and results on suboptimality in MPC~\cite[Section 2.7]{rawlings2020model}.
\end{remark}

\begin{remark}
Theorem~\ref{thm:stab} requires a number of assumptions, most of which are not too restrictive when proving closed-loop stability of a linearization-based MPC scheme utilizing the benefits of an artificial setpoint:
Assumption~\ref{ass:C2} (smoothness) is clearly required for a linearization-based MPC scheme.
Further, Assumption~\ref{ass:unique_steady_state} is a standard condition in the literature on tracking MPC, compare~\cite{rawlings2020model,limon2018nonlinear}.
If Assumption~\ref{ass:unique_steady_state} does not hold, we can still guarantee asymptotic stability of some steady-state, but not necessarily convergence to $x^{\rms\rmr}$.
Assumptions on controllability (Assumption~\ref{ass:ctrb}) are also standard in the presence of terminal equality constraints, see~\cite{rawlings2020model}.
In order to ensure that the employed bounds hold uniformly, it is crucial to make assumptions on compactness of the steady-state manifold (Assumption~\ref{ass:steady_state_compact}).
Moreover, as we show in Appendix~\ref{app:cond_ass}, Assumption~\ref{ass:convex_steady_state_manifold} holds, in fact, as long as $m=p$ and $y^\rmr$ is reachable.
In the absence of Assumption~\ref{ass:convex_steady_state_manifold}, we can still guarantee closed-loop stability of the optimal reachable steady-state of \emph{some} linearization, compare Inequality~\eqref{eq:thm_stab_ineq}, which may not necessarily be optimal for the nonlinear dynamics.
On the other hand, Assumption~\ref{ass:exist_equil_state} (non-singular dynamics) might possibly be relaxed, compare the discussion after Assumption~\ref{ass:exist_equil_state}.
Finally, we note that these assumptions hold in many practical applications, e.g., for the CSTR example we consider in Section~\ref{sec:example}.
\end{remark}

To conclude, the proposed MPC scheme based on repeatedly solving the linear MPC problem~\eqref{eq:MPC} leads to desirable closed-loop guarantees when applied to a nonlinear system, and the scheme can be tuned based on a \emph{single} design parameter $S$, which allows for a trade-off between the size of the region of attraction and the convergence speed.
Compared to nonlinear tracking MPC schemes such as~\cite{limon2018nonlinear,koehler2020nonlinear}, our approach has the drawback that convergence may be slower since $S$ needs to be chosen sufficiently small.
On the other hand, Problem~\eqref{eq:MPC} is a convex QP which can be solved up to global optimality very efficiently.
Further, the prediction model only requires an accurate description of the system close to the steady-state manifold, which may be simpler to obtain than a globally accurate model which is required to obtain superior performance with existing nonlinear MPC approaches.
Finally, as we show in our companion paper~\cite{berberich2021linearpart2}, the presented idea can be extended in order to develop a data-driven MPC scheme to control unknown nonlinear systems by continuously updating the measured data used for prediction.

\section{Numerical Example}\label{sec:example}
We apply the proposed MPC scheme to the CSTR from~\cite{mayne2011tube} with the nonlinear system dynamics $f(x,u)$ equal to
\begin{align*}
\begin{bmatrix}x_1+\frac{T_\rms}{\theta}(1-x_1)-T_\rms\bar{k}x_1e^{-\frac{M}{x_2}}\\
x_2+\frac{T_\rms}{\theta}(x_\rmf-x_2)+T_\rms\bar{k}x_1e^{-\frac{M}{x_2}}-T_\rms\alpha u(x_2-x_\rmc)
\end{bmatrix}.
\end{align*}
The states $x_1$ and $x_2$ are the temperature and the concentration, respectively, and the control input $u$ is the coolant flow rate.
These dynamics are obtained from the continuous-time dynamics in~\cite{mayne2011tube} via a simple Euler discretization with sampling time $T_\rms=0.2$.
The other parameters appearing in the vector field are $\theta=20$, $\bar{k}=300$, $M=5$, $x_\rmf=0.3947$, $x_\rmc=0.3816$, $\alpha=0.117$.
Our control goal is tracking of the output setpoint $y^\rmr=0.6519$ for the concentration, i.e., $h(x,u)=x_2$, while satisfying the input constraints $u_k\in\mathbb{U}=[0.1,2]$ for $k\in\mathbb{I}_{\geq0}$.
In order to set up the MPC, we consider the cost matrices $Q=I$, $R=0.05$, $S=100$, the prediction horizon $N=40$, and the input equilibrium constraints $\mathbb{U}^\rms=[0.11,1.99]$.
Since the dynamics are not of the control-affine form~\eqref{eq:nl_sys}, we implement the MPC scheme with an incremental input formulation $\Delta u_k\coloneqq u_{k+1}-u_k$ and include an additional penalty $\lVert\Delta u_k\rVert_2^2$ in the cost.

We first investigate whether our assumptions are met by the CSTR.
For this verification, we only consider linearization points in the relevant operating range, i.e., $\tilde{x}\in(0,1]^2$.
It is simple to verify that the considered system satisfies Assumption~\ref{ass:C2} (smoothness).
While Assumption~\ref{ass:exist_equil_state} (non-singular dynamics) does not hold due to the integrator dynamics $u_{k+1}=u_k+\Delta u_k$, our theoretical results still apply since Assumption~\ref{ass:exist_equil_state} holds for the original system (without $\Delta u$) and thus, for any given $(u^\rms,\Delta u^\rms)=(u^\rms,0)$ there still exists a unique steady-state $x^\rms$.
Assumption~\ref{ass:unique_steady_state} holds in $(0,1]^2$ except for a neighborhood of $x_2=x_\rmc$.
Similarly, the linearized dynamics are controllable (Assumption~\ref{ass:ctrb}) on $(0,1]^2$ except in a neighborhood of $x_2=x_\rmc$ or $x_2=0$. 
Further, Assumption~\ref{ass:steady_state_compact} (compact steady-state manifold of the linearization) clearly holds on the set $\tilde{x}\in(0,1]^2$ due to Assumption~\ref{ass:exist_equil_state}.
Finally, Assumption~\ref{ass:convex_steady_state_manifold} holds since $m=p$ and the setpoint $y^\rmr$ is reachable, i.e., Assumption~\ref{ass:reachability} holds, compare Appendix~\ref{app:cond_ass}.

\begin{figure}
\begin{center}
\includegraphics[width=0.5\textwidth]{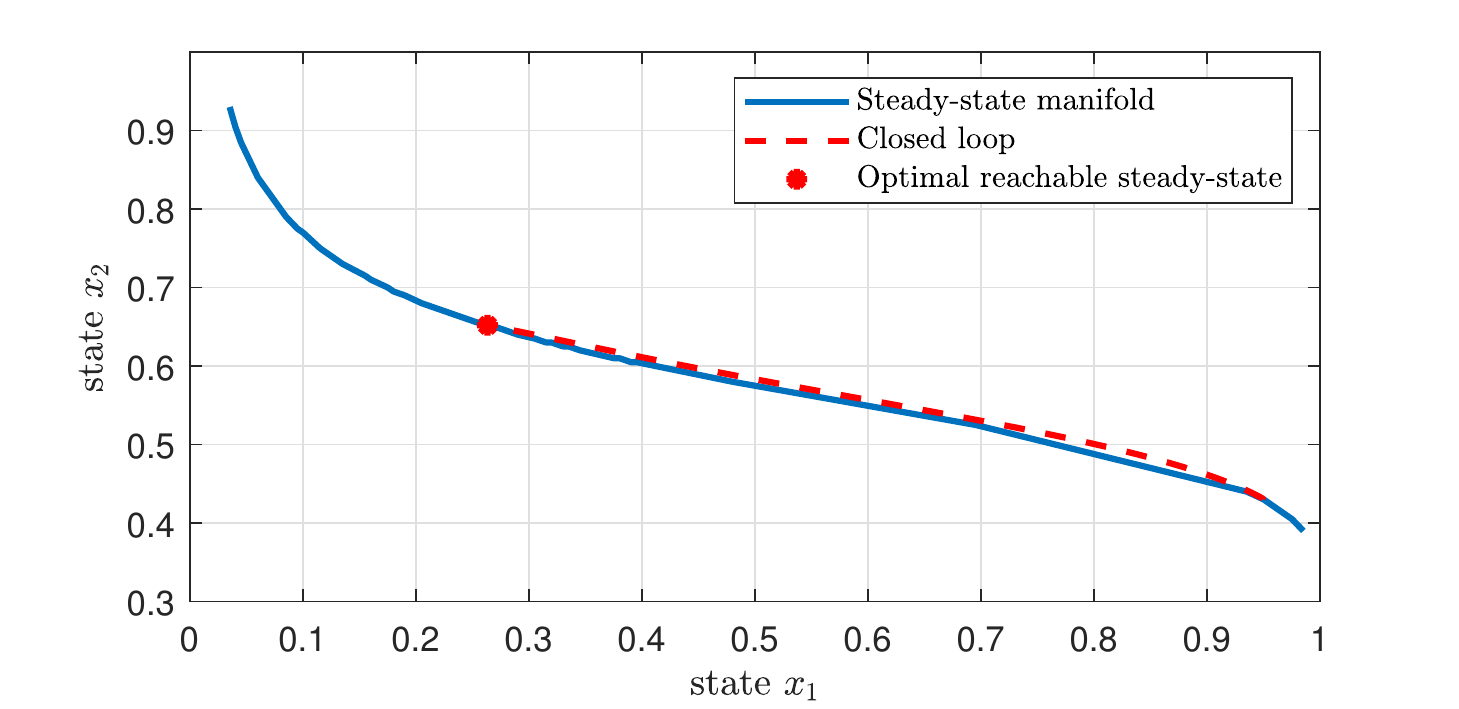}
\end{center}
\caption{State component of the steady-state manifold $\mathcal{Z}^\rms$ and closed-loop trajectory under the $n$-step MPC scheme (Algorithm~\ref{alg:MPC}) for the numerical example in Section~\ref{sec:example}.}
\label{fig:manifold_example}
\end{figure}

Figure~\ref{fig:manifold_example} shows the closed-loop state trajectory under Algorithm~\ref{alg:MPC} when starting at the initial state $x_0=\begin{bmatrix}0.9492&0.43\end{bmatrix}^\top$.
During the full closed-loop operation, the trajectory remains close to $\mathcal{Z}^\rms$ such that the prediction error induced by the linearization is small and $y_t$ asymptotically converges to $y^\rmr$.
For comparison, we also apply the following MPC schemes, each with terminal equality constraints, online optimization of an artificial equilibrium, an incremental input penalty, and the same design parameters as above:
\begin{enumerate}
\item the proposed MPC scheme in a one-step fashion (i.e., Algorithm~\ref{alg:MPC} with $n=1$),

\item a one-step MPC scheme using an LTI prediction model obtained by linearizing the nonlinear dynamics~\eqref{eq:nl_sys} at $x^{\rms\rmr}$ (called ``LTI-MPC''),

\item a one-step MPC scheme using an LTV prediction model obtained by linearizing the nonlinear dynamics~\eqref{eq:nl_sys} at time $t$ along the candidate solution $\bar{x}_1^*(t-1)$, $\bar{x}_2^*(t-1)$, $\dots$, $\bar{x}_N^*(t-1)$ (called ``LTV-MPC''), analogously to~\cite{cannon2011robust} and comparable to the real-time iteration scheme~\cite{diehl2005nominal}, and

\item the nonlinear tracking MPC scheme from~\cite{koehler2020nonlinear}.
\end{enumerate} 
The closed-loop state- and input-trajectories can be seen in Figure~\ref{fig:example_sol}.
First, note that, except for the LTI-MPC, all MPC schemes achieve asymptotic tracking of the desired setpoint.
The nonlinear tracking MPC from~\cite{koehler2020nonlinear} performs better than the LTV-MPC, which in turn outperforms the proposed $n$-step MPC scheme (Algorithm~\ref{alg:MPC}) as well as the corresponding $1$-step MPC scheme.
%
%
\begin{table}\label{tab:computation_times}
\begin{center}
\begin{tabular}{c|c|c|c}
&Setup QP&Optimization&Sum\\\hline
Nonlinear MPC&--&$21.6$&$21.6$\\\hline
LTV-MPC&$2.2$&$6.8$&$9$\\\hline
Proposed $1$- or $n$-step MPC&$0.6$&$6.8$&$7.4$
\end{tabular}
\vskip5pt
\caption{Average computation times in milliseconds of the MPC schemes in the numerical example.}
\end{center}
\end{table}
Table~I lists the times required for setting up the QPs, including the computation of the Jacobians, and solving the optimization problems arising in the considered MPC schemes (using 'quadprog' for the QPs and CasADi~\cite{andersson2019mathematical} with solver 'IPOPT' for the nonlinear optimization problem in~\cite{koehler2020nonlinear}).
Note that the LTV-MPC has slightly larger computation times since, at each time step, $N$ linearized dynamics need to be computed, whereas the proposed MPC scheme only requires the linearization at $x_t$.
While the LTV-MPC provides a good trade-off between computational complexity and closed-loop performance, theoretical results in the literature require either an additional bounding of the linearization error~\cite{cannon2011robust} or sufficiently many iterations~\cite{liao2020time,zanelli2021lyapunov}, compare Remark~\ref{rk:RTI}.
Without online optimization of an artificial setpoint, all MPC schemes considered above are initially infeasible.
Furthermore, note that the performance of the $1$-step MPC scheme is superior if compared to that of the $n$-step MPC scheme since the model is updated more frequently (twice as often) for the $1$-step scheme and hence, the influence of the prediction error is smaller.
Finally, the LTI-MPC based on the linearization at $x^{\rms\rmr}$ fails to track the desired setpoint.
To summarize, in application to a CSTR, the presented tracking MPC scheme using a linearized prediction model leads to a closed-loop performance which is comparable to that of nonlinear tracking MPC while being significantly more computationally efficient.

\begin{figure}
		\begin{center}
		\subfigure[State component $x_1$]
		{\includegraphics[width=0.5\textwidth]{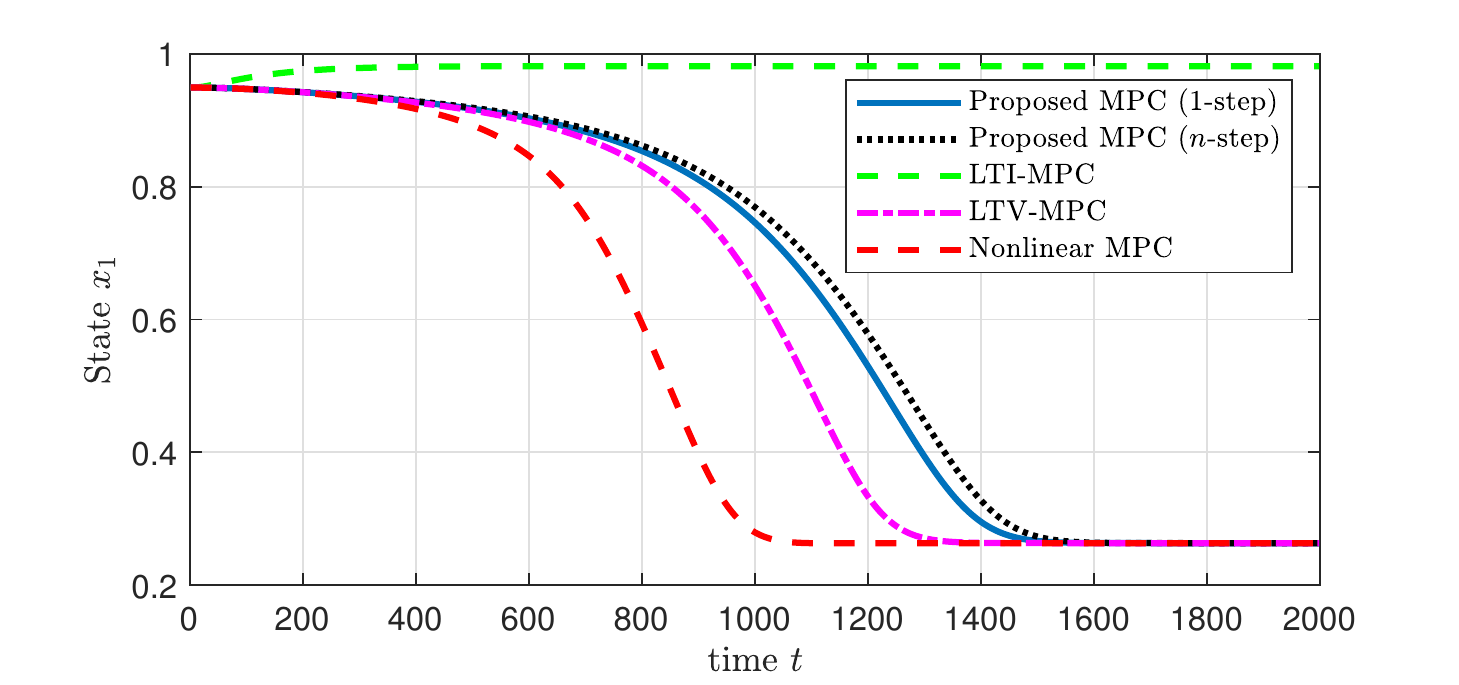}}
		\subfigure[State component $x_2$]
		{\includegraphics[width=0.5\textwidth]{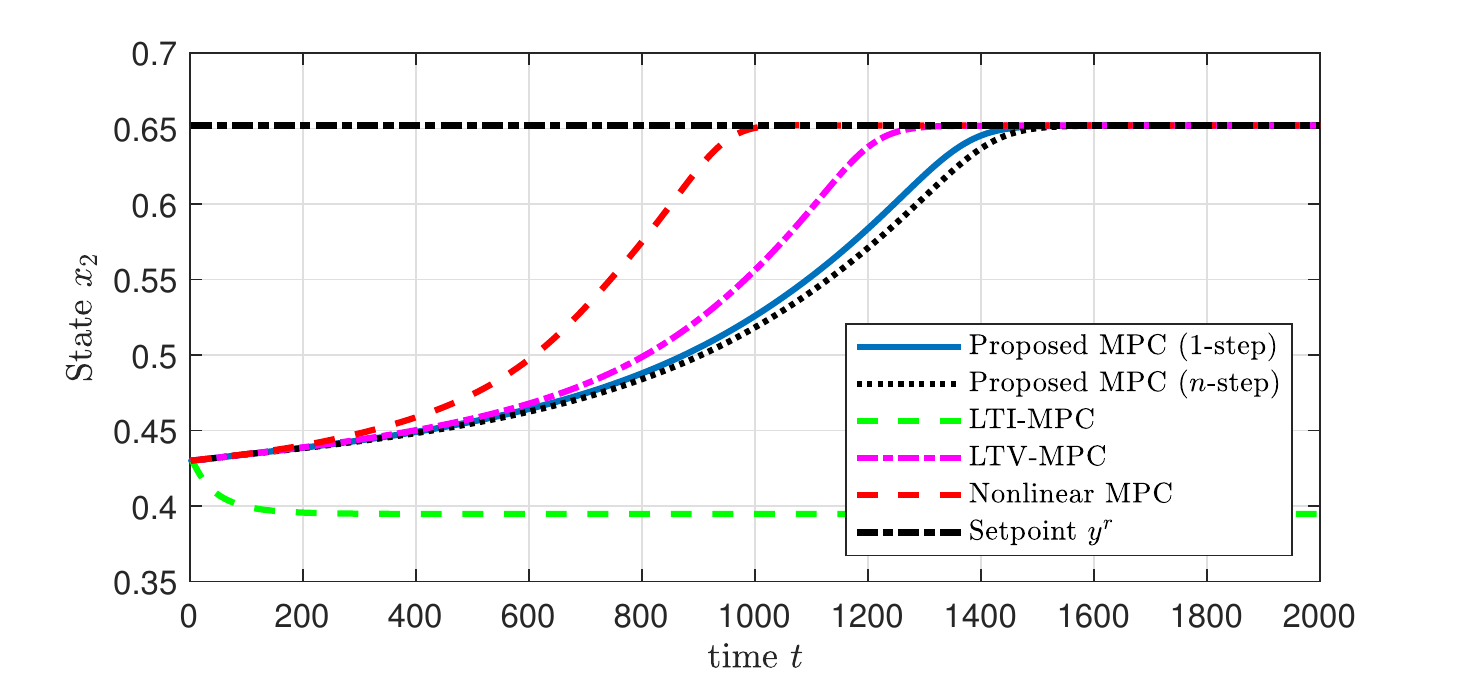}}
		\subfigure[Input component $u$]
		{\includegraphics[width=0.5\textwidth]{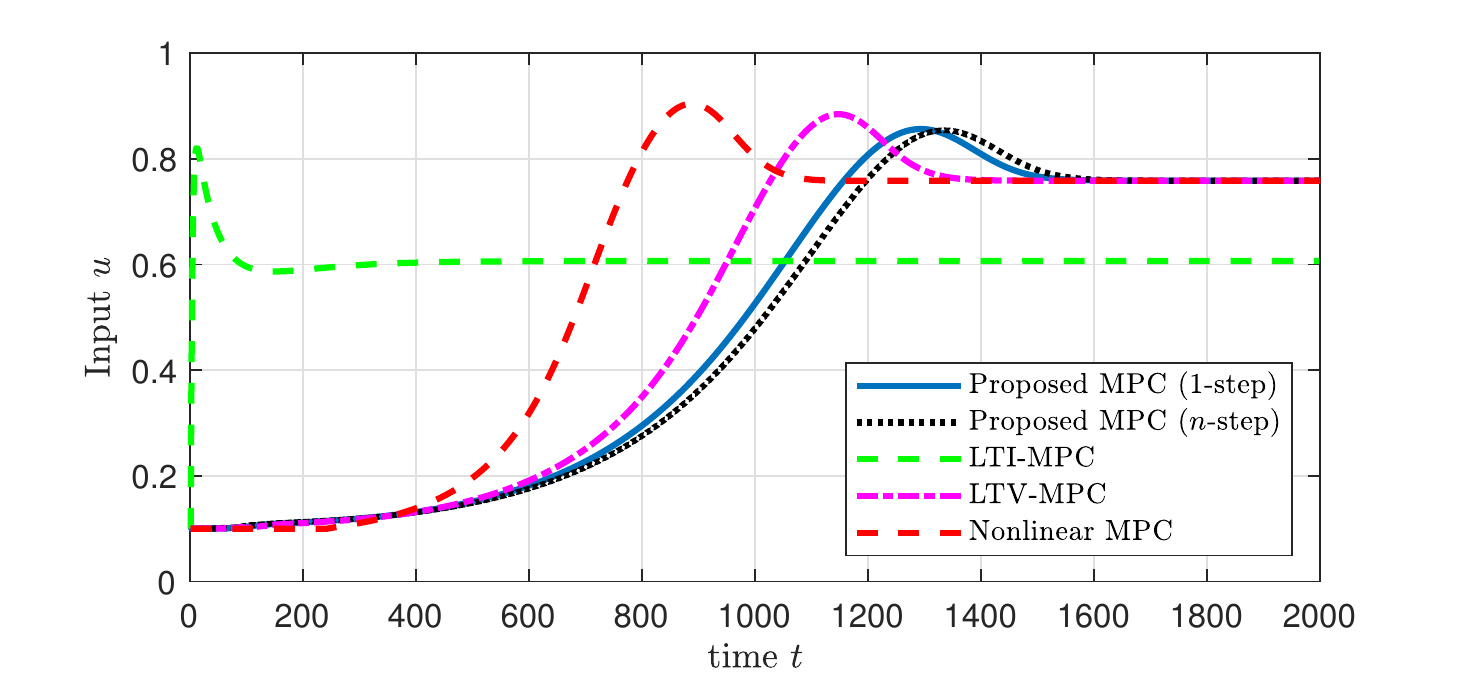}}
		\end{center}
		\caption{State components $x_1$ and $x_2$, and input component $u$ of the numerical example in Section~\ref{sec:example} are illustrated in Subfigures (a), (b) and (c), respectively, with linearization-based $1$-step MPC (solid), linearization-based $n$-step MPC (dotted), LTI-MPC (dashed), LTV-MPC (dash-dotted), and nonlinear MPC (dashed).}	\label{fig:example_sol}
\end{figure}

\section{Conclusion}\label{sec:conclusion}
In this paper, we presented a novel tracking MPC scheme for nonlinear systems using a prediction model based on the linearized dynamics at the current state.
As a key technical contribution, we proved that the optimal reachable equilibrium is exponentially stable in closed loop under reasonable assumptions on the underlying system.
The method was successfully applied to a numerical example where it achieved good performance while being computationally more efficient than a comparable nonlinear tracking MPC scheme.

The presented results build the basis for obtaining theoretical guarantees when using linear prediction models in MPC to control nonlinear systems.
In our companion paper~\cite{berberich2021linearpart2}, we exploit this viewpoint further to design an MPC scheme for unknown nonlinear systems with closed-loop stability guarantees based on linear data-dependent prediction models from behavioral systems theory~\cite{willems2005note}.

\bibliographystyle{IEEEtran}   
\bibliography{Literature}  

\section*{Appendix}
In the following, we provide a proof of Proposition~\ref{prop:rec_feas2} by considering two complementary cases with two different candidate solutions at time $t+n$.
First, in Appendix~\ref{appendix_A}, we prove the statement under the assumption that the tracking cost w.r.t. the artificial steady-state is large, quantified via a suitable inequality.
In Appendix~\ref{appendix_B}, we then consider the complementary case, which, together with the result in Appendix~\ref{appendix_A}, proves the full statement of Proposition~\ref{prop:rec_feas2}.
Finally, we present sufficient conditions for Assumption~\ref{ass:convex_steady_state_manifold} in Appendix~\ref{app:cond_ass}.
\renewcommand\thesection{\Alph{section}}
\setcounter{section}{0}
\section{Proof of Proposition~\ref{prop:rec_feas2} - candidate 1}\label{appendix_A}

\begin{proposition}\label{prop:rec_feas}
Suppose $N\geq n$ and Assumptions~\ref{ass:C2},~\ref{ass:unique_steady_state},~\ref{ass:ctrb},~\ref{ass:exist_equil_state}, and~\ref{ass:steady_state_compact} hold.
Then, there exist $V_{\max},J_{\mathrm{eq}}^{\max}>0$ such that, if $V(x_t)\leq V_{\max}$, $J_{\mathrm{eq},\rmlin}^*(x_t)\leq J_{\mathrm{eq}}^{\max}$, $J_{\mathrm{eq},\rmlin}^*(x_{t+n})\leq J_{\mathrm{eq}}^{\max}$, and there exists $\gamma_1>0$ such that 
\begin{align}\label{eq:prop_case1}
\begin{split}
\sum_{k=0}^{n-1}\lVert\bar{x}_k^*(t)-x^{\rms*}(t)&\rVert_2^2+\lVert\bar{u}_k^*(t)-u^{\rms*}(t)\rVert_2^2\\
&\geq\gamma_1\lVert x^{\rms*}(t)-x^{\rms\rmr}_{\rmlin}(x_t)\rVert_2^2,
\end{split}
\end{align}
then Problem~\eqref{eq:MPC} is feasible at time $t+n$ and there exists a constant $0<c_{\rmV1}<1$ such that $V(x_{t+n})\leq c_{\rmV1}V(x_t)$.
\end{proposition}
\begin{proof}
First, we define the candidate equilibrium $x^\rms\text{$'$}(t+n),u^\rms\text{$'$}(t+n)$ and the first $N-n$ components of the candidate input $\bar{u}'(t+n)$ (Part (i)).
Thereafter, in Part (ii), we derive useful bounds involving this candidate trajectory.
Next, we show that the state of this candidate solution at time $N-n$ is sufficiently close to the candidate artificial steady-state and thus, we can construct a local deadbeat controller to steer the system to this steady-state (Part (iii)).
Finally, we show $V(x_{t+n})\leq c_{\rmV1} V(x_t)$ in Part (iv).

Note that $x_t$ lies in the set $\{x\in\mathbb{R}^n\mid V(x)\leq V_{\max}\}$, which is compact due to the lower bound~\eqref{eq:lem_value_fcn_lower_bound} and Assumption~\ref{ass:steady_state_compact}.
We define the set $X$ as the union of the $N$-step reachable sets of the linearized and the nonlinear dynamics (compare~\cite[Definition 2]{schuermann2018reachset}), starting at $x_t$.
Using that the dynamics~\eqref{eq:nl_sys} are Lipschitz continuous and the input constraints are compact, we conclude that $X$ is compact.
Throughout this proof, whenever we apply Inequality~\eqref{eq:diff_lin}, we use the fact that all involved states lie in $X$ and we use the corresponding constant $c_X$.
\\
\textbf{(i) Definition of candidate solution for $k\in\mathbb{I}_{[0,N-n]}$}\\
We choose the candidate equilibrium input as the old solution, i.e., $u^\rms\text{$'$}(t+n)=u^{\rms*}(t)$.
According to Assumption~\ref{ass:exist_equil_state}, there exists a unique equilibrium state $x^\rms\text{$'$}(t+n)$ for the system linearized at $x_{t+n}$ such that
\begin{align*}
x^\rms\text{$'$}(t+n)=A_{x_{t+n}}x^\rms\text{$'$}(t+n)+Bu^\rms\text{$'$}(t+n)+e_{x_{t+n}}.
\end{align*}
The corresponding output $y^\rms\text{$'$}(t+n)$ is computed via~\eqref{eq:MPC5}.
Further, for $k\in\mathbb{I}_{[0,N-n-1]}$, we choose the candidate input as the previously optimal one, i.e., $\bar{u}_k'(t+n)=\bar{u}_{k+n}^*(t)$.
This leads to the state trajectory candidate $\bar{x}'_{k}(t+n)$, $k\in\mathbb{I}_{[0,N-n]}$, resulting from an open-loop application of $\bar{u}'(t+n)$ with initial condition $x_{t+n}$ to the dynamics linearized at time $t+n$, i.e., $\bar{x}'_0(t+n)=x_{t+n}$ and
\begin{align*}
\bar{x}'_{k+1}(t+n)&=f_{x_{t+n}}(\bar{x}_k'(t+n),\bar{u}_k'(t+n))\\
&=A_{x_{t+n}}\bar{x}'_{k}(t+n)+B\bar{u}_k'(t+n)+e_{x_{t+n}},
\end{align*}
for $k\in\mathbb{I}_{[0,N-n-1]}$.\\
\textbf{(ii) Bounds on candidate solution}\\
Throughout Part (ii) of the proof, let $k\in\mathbb{I}_{[0,N-n]}$.
Further, abbreviate $\underline{q}\coloneqq\lambda_{\min}(Q),\>\>\bar{q}\coloneqq\lambda_{\max}(Q)$, and similarly for $\underline{s}$, $\bar{s}$, $\underline{r}$, $\bar{r}$.
It clearly holds that
\begin{align}\label{eq:prop_proof_Lyap_sls2}
&\sum_{k=0}^{N-1}\lVert \bar{x}_k^*(t)-x^{\rms*}(t)\rVert_Q^2\\\nonumber
\leq&\sum_{k=0}^{N-1}\lVert \bar{x}_k^*(t)-x^{\rms*}(t)\rVert_Q^2+\lVert y^{\rms*}(t)-y^\rmr\rVert_S^2-J_{\mathrm{eq},\rmlin}^*(x_t)\\\nonumber
\leq &J_N^*(x_t)-J_{\mathrm{eq},\rmlin}^*(x_t)= V(x_t),
\end{align}
and hence,
\begin{align}\label{eq:prop_proof_Lyap_sls}
\sum_{k=0}^{N-1}\lVert \bar{x}_k^*(t)-x^{\rms*}(t)\rVert_2^2\leq\frac{1}{\underline{q}}V(x_t).
\end{align}
We bound now several expressions involving the optimal solution at time $t$ and the candidate solution at time $t+n$.\\
\textbf{(ii.a) Bound on $\lVert\bar{x}_n^*(t)-x_{t+n}\rVert_2$}\\
Using~\eqref{eq:diff_lin}, which holds by Assumption~\ref{ass:C2}, and~\eqref{eq:f_equals_f_lin}, we obtain
\begin{align}\label{eq:prop_proof_xn_xtn}
&\lVert\bar{x}_n^*(t)-x_{t+n}\rVert_2\\\nonumber
=&\lVert f_{x_t}(\bar{x}_{n-1}^*(t),\bar{u}_{n-1}^*(t))-f_{x_{t+n-1}}(x_{t+n-1},\bar{u}_{n-1}^*(t))\rVert_2\\\nonumber
\stackrel{\eqref{eq:diff_lin}}{\leq}&c_X\lVert\bar{x}_{n-1}^*(t)-x_t\rVert_2^2+L_f\lVert\bar{x}_{n-1}^*(t)-x_{t+n-1}\rVert_2\\\nonumber
\stackrel{\eqref{eq:ab_ineq}}{\leq}&2c_X\lVert\bar{x}_{n-1}^*(t)-x^{\rms*}(t)\rVert_2^2+2c_X\lVert x^{\rms*}(t)-x_t\rVert_2^2\\\nonumber
&+L_f\lVert\bar{x}_{n-1}^*(t)-x_{t+n-1}\rVert_2\\\nonumber
\stackrel{\eqref{eq:prop_proof_Lyap_sls}}{\leq}&2\frac{c_X}{\underline{q}}V(x_t)+L_f\lVert\bar{x}_{n-1}^*(t)-x_{t+n-1}\rVert_2\\\nonumber
\leq&\dots\leq2\frac{c_X}{\underline{q}}V(x_t)\sum_{k=0}^{n-2}L_f^k,
\end{align}
where the summand for $k=n-1$ vanishes since $\lVert\bar{x}_1^*(t)-x_{t+1}\rVert_2=0$.\\
\textbf{(ii.b) Bound on $\lVert\bar{x}_k'(t+n)-\bar{x}_{k+n}^*(t)\rVert_2$}\\
Define $\{a_k\}_{k=0}^{N-n}$ recursively in dependence of $V(x_t)$ as
\begin{align*}
a_0&\coloneqq2\frac{c_X}{\underline{q}}V(x_t)\sum_{k=0}^{n-2}L_f^k,\\
a_k&\coloneqq2c_X a_{k-1}^2+L_fa_{k-1}+16\frac{c_X}{\underline{q}}V(x_t)\\\nonumber
&\quad+8c_X\left(2\frac{c_X}{\underline{q}}\sum_{k=0}^{n-2}L_f^k\right)^2V(x_t)^2,
\>\> k=1,\dots,N-n.
\end{align*}
In the following, we prove that for any $k\in\mathbb{I}_{[0,N-n]}$
\begin{align}\label{eq:prop_proof_diff_opt}
\lVert \bar{x}_k'(t+n)-\bar{x}_{k+n}^*(t)\rVert_2\leq a_k.
\end{align}
According to~\eqref{eq:prop_proof_xn_xtn} and using $\bar{x}_0'(t+n)=x_{t+n}$, Inequality~\eqref{eq:prop_proof_diff_opt} holds for $k=0$.
Using an induction argument over $k$, we have
\begin{align}\nonumber
&\quad\lVert\bar{x}_k'(t+n)-\bar{x}_{k+n}^*(t)\rVert_2\\\nonumber
&=\lVert f_{x_{t+n}}(\bar{x}_{k-1}'(t+n),\bar{u}_{k+n-1}^*(t))\\\nonumber
&\quad\qquad\quad\qquad-f_{x_t}(\bar{x}_{k+n-1}^*(t),\bar{u}_{k+n-1}^*(t))\rVert_2\\\nonumber
&\stackrel{\eqref{eq:diff_lin}}{\leq}c_X\lVert\bar{x}_{k-1}'(t+n)-x_{t+n}\rVert_2^2
+c_X\lVert\bar{x}_{k+n-1}^*(t)-x_t\rVert_2^2\\\nonumber
&\quad+L_f\lVert\bar{x}_{k-1}'(t+n)-\bar{x}_{k+n-1}^*(t)\rVert_2\\\nonumber
&\stackrel{\eqref{eq:ab_ineq}}{\leq} 2c_X\lVert\bar{x}_{k-1}'(t+n)-\bar{x}_{k+n-1}^*(t)\rVert_2^2\\\nonumber
&\quad+L_f\lVert\bar{x}_{k-1}'(t+n)-\bar{x}_{k+n-1}^*(t)\rVert_2\\\nonumber
&\quad+2c_X\lVert\bar{x}_{k+n-1}^*(t)-x_{t+n}\rVert_2^2+c_X\lVert\bar{x}_{k+n-1}^*(t)-x_t\rVert_2^2\\\nonumber
&\stackrel{\eqref{eq:ab_ineq},\eqref{eq:prop_proof_diff_opt}}{\leq}
2c_X a_{k-1}^2+L_fa_{k-1}+6c_X\lVert\bar{x}_{k+n-1}^*(t)-x^{\rms*}(t)\rVert_2^2\\\nonumber
&\quad+4c_X\lVert x^{\rms*}(t)-x_{t+n}\rVert_2^2+2c_X\lVert x^{\rms*}(t)-x_t\rVert_2^2\\\nonumber
&\stackrel{\eqref{eq:ab_ineq},\eqref{eq:prop_proof_Lyap_sls}}{\leq}
2c_X a_{k-1}^2+L_fa_{k-1}+8\frac{c_X}{\underline{q}}V(x_t)\\\nonumber
&\quad+8c_X\lVert x^{\rms*}(t)-\bar{x}_n^*(t)\rVert_2^2+8c_X\lVert\bar{x}_n^*(t)-x_{t+n}\rVert_2^2\\\nonumber
&\stackrel{\eqref{eq:prop_proof_Lyap_sls},\eqref{eq:prop_proof_xn_xtn}}{\leq}
2c_X a_{k-1}^2+L_fa_{k-1}+16\frac{c_X}{\underline{q}}V(x_t)\\\nonumber
&\qquad+8c_X\left(2\frac{c_X}{\underline{q}}\sum_{k=0}^{n-2}L_f^k\right)^2V(x_t)^2=a_k,
\end{align}
which proves~\eqref{eq:prop_proof_diff_opt}.
Note that $a_k$ is a polynomial in $V(x_t)$ which becomes arbitrarily small if $V(x_t)$ is sufficiently small.\\
\textbf{(ii.c) Bound on $\lVert x^{\rms*}(t)-x^\rms\text{$'$}(t+n)\rVert_2$}\\
Note that
\begin{align*}
(I-A_{x_t})x^{\rms*}(t)&=Bu^{\rms*}(t)+e_{x_t},\\
(I-A_{x_{t+n}})x^\rms\text{$'$}(t+n)&=Bu^\rms\text{$'$}(t+n)+e_{x_{t+n}}.
\end{align*}
Using additionally $u^\rmsp(t+n)=u^{\rms*}(t)$,
this implies
\begin{align*}
&(I-A_{x_{t+n}})x^\rms\text{$'$}(t+n)\\
=&x^{\rms*}(t)-x^{\rms*}(t)+Bu^\rms\text{$'$}(t+n)+e_{x_{t+n}}\\
=&(I-A_{x_t})x^{\rms*}(t)+e_{x_{t+n}}-e_{x_t}.
\end{align*}
Using that $(I-A_{x_{t+n}})$ is invertible by Assumption~\ref{ass:exist_equil_state}, we obtain
\begin{align*}
x^\rms\text{$'$}(t+n)&=(I-A_{x_{t+n}})^{-1}\big((I-A_{x_t})x^{\rms*}(t)+e_{x_{t+n}}-e_{x_t}\big)\\
&=x^{\rms*}(t)+(I-A_{x_{t+n}})^{-1}\big((A_{x_{t+n}}-A_{x_t})x^{\rms*}(t)\\
&\quad+e_{x_{t+n}}-e_{x_t}\big).
\end{align*}
Moreover, Assumption~\ref{ass:exist_equil_state} implies $\lVert (I-A_{x_{t+n}})^{-1}\rVert_2\leq\frac{1}{\underline{\sigma}}$ and hence, we arrive at
\begin{align}\label{eq:prop_proof_equil_diff_aux}
&\lVert x^\rms\text{$'$}(t+n)-x^{\rms*}(t)\rVert_2\\\nonumber
\leq&\frac{1}{\underline{\sigma}}\lVert f_{x_{t+n}}(x^{\rms*}(t),u^{\rms*}(t))-f_{x_t}(x^{\rms*}(t),u^{\rms*}(t))\rVert_2\\\nonumber
\stackrel{\eqref{eq:diff_lin}}{\leq}&\frac{c_X}{\underline{\sigma}}
\left(\lVert x^{\rms*}(t)-x_{t+n}\rVert_2^2+\lVert x^{\rms*}(t)-x_t\rVert_2^2\right).
\end{align}
Together with~\eqref{eq:prop_proof_Lyap_sls} and~\eqref{eq:prop_proof_xn_xtn}, this implies
\begin{align}\label{eq:prop_proof_equil_diff}
&\lVert x^\rms\text{$'$}(t+n)-x^{\rms*}(t)\rVert_2\\\nonumber
\stackrel{\eqref{eq:ab_ineq},\eqref{eq:prop_proof_Lyap_sls}}{\leq}&\frac{c_X}{\underline{\sigma}}(2\lVert x^{\rms*}(t)-\bar{x}_n^*(t)\rVert_2^2+2\lVert\bar{x}_n^*(t)-x_{t+n}\rVert_2^2+\frac{V(x_t)}{\underline{q}})\\\nonumber
\stackrel{\eqref{eq:prop_proof_Lyap_sls}}{\leq}&\frac{c_X}{\underline{\sigma}}\left(3\frac{V(x_t)}{\underline{q}}+2\lVert\bar{x}_n^*(t)-x_{t+n}\rVert_2^2\right)\\\nonumber
\stackrel{\eqref{eq:prop_proof_xn_xtn}}{\leq}&\frac{c_X}{\underline{\sigma}}\left(
3\frac{V(x_t)}{\underline{q}}+2\left(2\frac{c_X}{\underline{q}}\sum_{k=0}^{n-2}L_f^k\right)^2V(x_t)^2\right)\\\nonumber
\eqqcolon &c_1 V(x_t)^2+c_2V(x_t).
\end{align}
\textbf{(ii.d) Bound on $\lVert y^\rms\text{$'$}(t+n)-y^{\rms*}(t)\rVert_2$}\\
By using an inequality of the form~\eqref{eq:diff_lin} for the vector field $h$ (note that $h$ is sufficiently smooth by Assumption~\ref{ass:C2}), there exist constants $c_{Xh},L_h\geq0$ such that
\begin{align*}
&\lVert y^\rms\text{$'$}(t+n)-y^{\rms*}(t)\rVert_2\\
=&\lVert h_{x_{t+n}}(x^\rms\text{$'$}(t+n),u^{\rms*}(t))-h_{x_t}(x^{\rms*}(t),u^{\rms*}(t))\rVert_2\\
\leq&c_{Xh}\lVert x^\rms\text{$'$}(t+n)-x_{t+n}\rVert_2^2+c_{Xh}\lVert x^{\rms*}(t)-x_t\rVert_2^2\\
&+L_h\lVert x^\rms\text{$'$}(t+n)-x^{\rms*}(t)\rVert_2\\
\stackrel{\eqref{eq:prop_proof_Lyap_sls},\eqref{eq:prop_proof_equil_diff}}{\leq}&c_{Xh}\lVert x^\rms\text{$'$}(t+n)-x_{t+n}\rVert_2^2+\frac{c_{Xh}}{\underline{q}}V(x_t)\\
&+L_h(c_1V(x_t)^2+c_2V(x_t)).
\end{align*}
Moreover, using $\lVert a+b+c\rVert_2^2\leq2\lVert a\rVert_2^2+4\lVert b\rVert_2^2+4\lVert c\rVert_2^2$, which holds for arbitrary $a$, $b$, $c$ due to~\eqref{eq:ab_ineq}, we obtain
\begin{align*}
&\lVert x^\rms\text{$'$}(t+n)-x_{t+n}\rVert_2^2\leq2\lVert x^\rms\text{$'$}(t+n)-x^{\rms*}(t)\rVert_2^2\\
&+4\lVert x^{\rms*}(t)-\bar{x}_{n}^*(t)\rVert_2^2+4\Vert\bar{x}_n^*(t)-x_{t+n}\rVert_2^2\\
&\stackrel{\eqref{eq:prop_proof_Lyap_sls},\eqref{eq:prop_proof_xn_xtn},\eqref{eq:prop_proof_equil_diff}}{\leq}
2(c_1V(x_t)^2+c_2V(x_t))^2+\frac{4}{\underline{q}}V(x_t)\\
&\qquad\qquad+4\left(2\frac{c_X}{\underline{q}}V(x_t)\sum_{k=0}^{n-2}L_f^k\right)^2.
\end{align*}
Hence, using $V(x_t)\leq V_{\max}$, there exists $c_3>0$ such that
\begin{align}\label{eq:prop_proof_equil_diff_output}
\lVert y^\rms\text{$'$}(t+n)-y^{\rms*}(t)\rVert_2\leq c_3V(x_t).
\end{align}
\\
\textbf{(iii) Appending deadbeat controller}\\
In the following, we show that for $V_{\max}$ sufficiently small $x^\rms\text{$'$}(t+n)$ is sufficiently close to $\bar{x}_{N-n}'(t+n)$ such that we can append a deadbeat controller steering the state to $x^\rms\text{$'$}(t+n)$ in $n$ steps.
To be precise, combining~\eqref{eq:prop_proof_diff_opt} with $k=N-n$ and~\eqref{eq:prop_proof_equil_diff}, we obtain
\begin{align}\label{eq:prop_proof_traj_vs_equil}
&\lVert\bar{x}_{N-n}'(t+n)-x^\rms\text{$'$}(t+n)\rVert_2\\\nonumber
\leq&\lVert\bar{x}_{N-n}'(t+n)-\bar{x}_{N}^*(t)\rVert_2+\lVert\bar{x}_{N}^*(t)-x^{\rms*}(t)\rVert_2\\\nonumber
&+\lVert x^{\rms*}(t)-x^\rms\text{$'$}(t+n)\rVert_2\\\nonumber
\stackrel{\eqref{eq:prop_proof_diff_opt},\eqref{eq:prop_proof_equil_diff}}{\leq}&a_{N-n}+c_1V(x_t)^2+c_2V(x_t),
\end{align}
where for the second inequality we used that $\bar{x}_N^*(t)=x^{\rms*}(t)$ due to the terminal equality constraint~\eqref{eq:MPC2}.
Assumption~\ref{ass:ctrb} implies the existence of an input $\bar{u}_k'(t+n),k\in\mathbb{I}_{[N-n,N-1]}$, steering the state to $\bar{x}_N'(t+n)=x^\rms\text{$'$}(t+n)$ while satisfying
\begin{align}\label{eq:prop_proof_traj_vs_equil2}
&\sum_{k=N-n}^{N-1}\lVert\bar{x}_k'(t+n)-x^\rms\text{$'$}(t+n)\rVert_2\\\nonumber
&\qquad\>\>\>+\lVert\bar{u}_k'(t+n)-u^\rms\text{$'$}(t+n)\rVert_2\\\nonumber
&\stackrel{\eqref{eq:ass_ctrb}}{\leq}\Gamma\lVert \bar{x}_{N-n}'(t+n)-x^\rms\text{$'$}(t+n)\rVert_2\\\nonumber
&\stackrel{\eqref{eq:prop_proof_traj_vs_equil}}{\leq}\Gamma(a_{N-n}+c_1V(x_t)^2+c_2V(x_t)).
\end{align}
If $V_{\max}$ and hence $V(x_t)$ and $a_{N-n}$ are sufficiently small, then $\bar{u}_k'(t+n)\in\mathbb{U}$ for $k\in\mathbb{I}_{[N-n,N-1]}$ (note that $u^\rms\text{$'$}(t+n)\in\text{int}(\mathbb{U})$), i.e., the candidate input satisfies the input constraints.\\
\textbf{(iv) Invariance of $V(x_t)\leq V_{\max}$}\\
So far, we have only shown that the MPC scheme is feasible at time $t+n$.
It remains to be shown that there exists a constant $0<c_{\rmV1}<1$ such that $V(x_{t+n})\leq c_{\rmV1} V(x_t)$.
Note that
\begin{align}\label{eq:prop_proof_stage_cost_bound1}
&J_N^*(x_{t+n})-J_N^*(x_t)\\\nonumber
\leq&\sum_{k=0}^{N-1}\lVert \bar{x}_k'(t+n)-x^\rms\text{$'$}(t+n)\rVert_Q^2+\lVert y^\rms\text{$'$}(t+n)-y^\rmr\rVert_S^2\\\nonumber
&+\sum_{k=0}^{N-1}\lVert \bar{u}_k'(t+n)-u^\rms\text{$'$}(t+n)\rVert_R^2-\lVert y^{\rms*}(t)-y^\rmr\rVert_S^2\\\nonumber
&-\sum_{k=0}^{N-1}(\lVert\bar{x}_k^*(t)-x^{\rms*}(t)\rVert_Q^2+\lVert\bar{u}_k^*(t)-u^{\rms*}(t)\rVert_R^2).
\end{align}
We now bound several terms on the right-hand side of~\eqref{eq:prop_proof_stage_cost_bound1} separately.
The definition of the input candidate implies
\begin{align*}
&\sum_{k=0}^{N-1}\lVert \bar{u}_k'(t+n)-u^\rms\text{$'$}(t+n)\rVert_R^2-\lVert \bar{u}_k^*(t)-u^{\rms*}(t)\rVert_R^2\\
\stackrel{\eqref{eq:prop_proof_traj_vs_equil2}}{\leq}&-\sum_{k=0}^{n-1}\lVert\bar{u}_k^*(t)-u^{\rms*}(t)\rVert_R^2\\
&+\bar{r}\cdot\Gamma^2(a_{N-n}+c_1V(x_t)^2+c_2V(x_t))^2.
\end{align*}
Using $\lVert y^{\rms*}(t)-y^\rmr\rVert_S\leq \sqrt{J_N^*(x_t)}$, we obtain
\begin{align}\label{eq:prop_proof_dummy}
&\lVert y^\rms\text{$'$}(t+n)-y^\rmr\rVert_S^2-\lVert y^{\rms*}(t)-y^\rmr\rVert_S^2\\\nonumber
\stackrel{\eqref{eq:prop_proof_std_norm_ineq}}{\leq}&\lVert y^\rms\text{$'$}(t+n)-y^{\rms*}(t)\rVert_S^2\\\nonumber
&+2\lVert y^\rms\text{$'$}(t+n)-y^{\rms*}(t)\rVert_S\lVert y^{\rms*}(t)-y^\rmr\rVert_S\\\nonumber
\leq&\lVert y^\rms\text{$'$}(t+n)-y^{\rms*}(t)\rVert_S^2\\\nonumber
&+2\lVert y^\rms\text{$'$}(t+n)-y^{\rms*}(t)\rVert_S\sqrt{V(x_t)+J_{\mathrm{eq},\rmlin}^*(x_t)}\\\nonumber
\stackrel{\eqref{eq:prop_proof_equil_diff_output}}{\leq}&\bar{s}c_3^2V(x_t)^2+2\sqrt{\bar{s}}c_3V(x_t)\sqrt{V(x_t)+J_{\mathrm{eq},\rmlin}^*(x_t)}.
\end{align}
Finally, note that
\begin{align*}
&\sum_{k=0}^{N-1}\lVert\bar{x}_k'(t+n)-x^\rms\text{$'$}(t+n)\rVert_Q^2-\lVert\bar{x}_k^*(t)-x^{\rms*}(t)\rVert_Q^2\\
&=-\sum_{k=0}^{n-1}\lVert\bar{x}_k^*(t)-x^{\rms*}(t)\rVert_Q^2\\
&+\sum_{k=0}^{N-n-1}\lVert\bar{x}_k'(t+n)-x^\rms\text{$'$}(t+n)\rVert_Q^2-\lVert\bar{x}_{k+n}^*(t)-x^{\rms*}(t)\rVert_Q^2\\
&+\underbrace{\sum_{k=N-n}^{N-1}\lVert\bar{x}_k'(t+n)-x^\rms\text{$'$}(t+n)\rVert_Q^2}_{\stackrel{\eqref{eq:prop_proof_traj_vs_equil2}}{\leq}\bar{q}\Gamma^2(a_{N-n}+c_1V(x_t)^2+c_2V(x_t))^2}.
\end{align*}
We bound the second sum on the right-hand side further as
\begin{align*}
&\sum_{k=0}^{N-n-1}\lVert\bar{x}_k'(t+n)-x^\rms\text{$'$}(t+n)\rVert_Q^2-\lVert\bar{x}_{k+n}^*(t)-x^{\rms*}(t)\rVert_Q^2\\
&\stackrel{\eqref{eq:prop_proof_std_norm_ineq}}{\leq}\sum_{k=0}^{N-n-1}\lVert\bar{x}_k'(t+n)-x^\rms\text{$'$}(t+n)-\bar{x}_{k+n}^*(t)+x^{\rms*}(t)\rVert_Q^2\\
&\quad+2\lVert\bar{x}_k'(t+n)-x^\rms\text{$'$}(t+n)-\bar{x}_{k+n}^*(t)+x^{\rms*}(t)\rVert_Q\\
&\quad\quad\cdot\lVert\bar{x}_{k+n}^*(t)-x^{\rms*}(t)\rVert_Q\\
&\stackrel{\eqref{eq:ab_ineq},\eqref{eq:prop_proof_Lyap_sls2}}{\leq}\sum_{k=0}^{N-n-1}2\lVert\bar{x}_k'(t+n)-\bar{x}_{k+n}^*(t)\rVert_Q^2\\
&\quad+2\lVert x^{\rms*}(t)-x^\rms\text{$'$}(t+n)\rVert_Q^2\\
&\quad+2\sqrt{V(x_t)}\lVert\bar{x}_k'(t+n)-\bar{x}_{k+n}^*(t)\rVert_Q\\
&\quad+2\sqrt{V(x_t)}\lVert x^{\rms*}(t)-x^\rms\text{$'$}(t+n)\rVert_Q\\
&\stackrel{\eqref{eq:prop_proof_diff_opt},\eqref{eq:prop_proof_equil_diff}}{\leq}
\sum_{k=0}^{N-n-1}2\bar{q}(a_k^2+(c_1V(x_t)^2+c_2V(x_t))^2)\\
&\qquad+2\sqrt{\bar{q}V(x_t)}(a_k+c_1V(x_t)^2+c_2V(x_t)).
\end{align*}
Inserting all of the derived bounds into~\eqref{eq:prop_proof_stage_cost_bound1}, we arrive at
\begin{align}\label{eq:prop_proof_stage_cost_bound2}
&J_N^*(x_{t+n})-J_N^*(x_t)\\\nonumber
\leq&-\sum_{k=0}^{n-1}(\lVert\bar{x}_k^*(t)-x^{\rms*}(t)\rVert_Q^2+\lVert\bar{u}_k^*(t)-u^{\rms*}(t)\rVert_R^2)\\\nonumber
&+\sum_{k=0}^{N-n-1}\Big(2\bar{q}(a_k^2+(c_1V(x_t)^2+c_2V(x_t))^2)\\\nonumber
&\qquad\qquad+2\sqrt{\bar{q}V(x_t)}(a_k+c_1V(x_t)^2+c_2V(x_t)\Big)\\\nonumber
&+(\bar{q}+\bar{r})\Gamma^2(a_{N-n}+c_1V(x_t)^2+c_2V(x_t))^2\\\nonumber
&+\bar{s}c_3^2V(x_t)^2+2\sqrt{\bar{s}}c_3V(x_t)\sqrt{V(x_t)+J_{\mathrm{eq},\rmlin}^*(x_t)}.
\end{align}
%
%
Note that~\eqref{eq:prop_case1} implies
\begin{align}\label{eq:prop_proof_stage_cost_bound3}
&-\sum_{k=0}^{n-1}(\lVert\bar{x}_k^*(t)-x^{\rms*}(t)\rVert_Q^2+\lVert\bar{u}_k^*(t)-u^{\rms*}(t)\rVert_R^2)\\\nonumber
&\stackrel{\eqref{eq:prop_case1}}{\leq}-\frac{1}{2}\sum_{k=0}^{n-1}(\lVert\bar{x}_k^*(t)-x^{\rms*}(t)\rVert_Q^2+\lVert\bar{u}_k^*(t)-u^{\rms*}(t)\rVert_R^2)\\\nonumber
&\quad-\frac{\gamma_1\min\{\underline{q},\underline{r}\}}{2}\lVert x^{\rms*}(t)-x^{\rms\rmr}_{\rmlin}(x_t)\rVert_2^2\\\nonumber
&\stackrel{\eqref{eq:ab_ineq}}{\leq}-\frac{\min\{\underline{q},\underline{r}\}\cdot\min\{1,\gamma_1\}}{4}\lVert x_t-x^{\rms\rmr}_{\rmlin}(x_t)\rVert_2^2.
\end{align}
The local upper bound~\eqref{eq:lem_value_fcn_upper_bound}, which holds for $\lVert x_t-x^{\rms\rmr}_{\rmlin}(x_t)\rVert_2\leq\delta$, implies that for all $x_t\in X$
\begin{align}\label{eq:prop_proof_lax_upper_bound}
V(x_t)\leq c_{\rmu,\rmV}\lVert x_t-x^{\rms\rmr}_{\rmlin}(x_t)\rVert_2^2,
\end{align}
where $c_{\rmu,\rmV}\coloneqq\max\{\frac{V_{\max}}{\delta^2},c_\rmu\}$.
Thus, we obtain
\begin{align}\label{eq:prop_proof_state_negative}
&-\sum_{k=0}^{n-1}(\lVert\bar{x}_k^*(t)-x^{\rms*}(t)\rVert_Q^2+\lVert\bar{u}_k^*(t)-u^{\rms*}(t)\rVert_R^2)\\\nonumber
&\stackrel{\eqref{eq:prop_proof_stage_cost_bound3},\eqref{eq:prop_proof_lax_upper_bound}}{\leq}-\frac{\min\{\underline{q},\underline{r}\}\cdot\min\{1,\gamma_1\}}{4c_{\rmu,\rmV}}V(x_t).
\end{align}
Note that all positive terms on the right-hand side of~\eqref{eq:prop_proof_stage_cost_bound2} are either at least of order $V(x_t)^{2}$, or they are of order $V(x_t)$ but are multiplied by $\sqrt{V(x_t)+J_{\mathrm{eq},\rmlin}^*(x_t)}$.
Hence, if we plug~\eqref{eq:prop_proof_state_negative} into~\eqref{eq:prop_proof_stage_cost_bound2} and choose $V_{\max}$ and $J_{\mathrm{eq}}^{\max}$ sufficiently small, then we obtain
\begin{align*}
J_N^*(x_{t+n})-J_N^*(x_{t})\leq(\tilde{c}_{\rmV1}-1)V(x_t),
\end{align*}
for some $0<\tilde{c}_{\rmV1}<1$.
This implies
\begin{align}\label{eq:prop_proof_V}
V(x_{t+n})=&J_N^*(x_{t+n})-J_{\mathrm{eq},\rmlin}^*(x_{t+n})\\\nonumber
\leq&J_N^*(x_t)+(\tilde{c}_{\rmV1}-1)V(x_t)-J_{\mathrm{eq},\rmlin}^*(x_{t+n})\\\nonumber
=&\tilde{c}_{\rmV1}V(x_t)+J_{\mathrm{eq},\rmlin}^*(x_t)-J_{\mathrm{eq},\rmlin}^*(x_{t+n}).
\end{align}
As the last step in the proof, we now derive a bound on $J_{\mathrm{eq},\rmlin}^*(x_t)-J_{\mathrm{eq},\rmlin}^*(x_{t+n})$.
First, we define a candidate solution to the optimization problem~\eqref{eq:opt_equil_prob_lin} with optimal cost $J_{\mathrm{eq},\rmlin}^*(x_t)$.
The input candidate is defined as $\tilde{u}^\rms=u^{\rms\rmr}_{\rmlin}(x_{t+n})$, i.e., the optimal reachable equilibrium input for the linearized system at $x_{t+n}$.
The state and output candidates are chosen as the corresponding equilibria for the dynamics linearized at $x_t$, i.e., 
\begin{align*}
\tilde{x}^\rms&=A_{x_t}\tilde{x}^\rms+B\tilde{u}^\rms+e_{x_t},\\
\tilde{y}^\rms&=C_{x_t}\tilde{x}^\rms+D\tilde{u}^\rms+r_{x_t}.
\end{align*}
Note that such $\tilde{x}^\rms$ (and hence also $\tilde{y}^\rms$) exists due to Assumption~\ref{ass:exist_equil_state}.
Following the same steps leading to~\eqref{eq:prop_proof_equil_diff_output}, it can be shown that there exists $\tilde{c}>0$ such that
\begin{align}\label{eq:prop_proof_output_equil_diff_end}
\lVert\tilde{y}^\rms-y^{\rms\rmr}_{\rmlin}(x_{t+n})\rVert_S\leq\tilde{c}V(x_t).
\end{align}
Hence, by optimality, we obtain
\begin{align}\label{eq:prop2_proof_Jeq_xtxtn}
&J_{\mathrm{eq},\rmlin}^*(x_t)-J_{\mathrm{eq},\rmlin}^*(x_{t+n})\\\nonumber
\leq&\lVert\tilde{y}^\rms-y^\rmr\rVert_S^2-\lVert y^{\rms\rmr}_{\rmlin}(x_{t+n})-y^\rmr\rVert_S^2\\\nonumber
\stackrel{\eqref{eq:prop_proof_std_norm_ineq}}{\leq}&\lVert \tilde{y}^\rms-y^{\rms\rmr}_{\rmlin}(x_{t+n})\rVert_S^2\\\nonumber
&+2\lVert\tilde{y}^\rms-y^{\rms\rmr}_{\rmlin}(x_{t+n})\rVert_S\lVert y^{\rms\rmr}_{\rmlin}(x_{t+n})-y^\rmr\rVert_S\\\nonumber
\stackrel{\eqref{eq:prop_proof_output_equil_diff_end}}{\leq}&\tilde{c}^2V(x_t)^2+2\tilde{c}V(x_t)\sqrt{J_{\mathrm{eq},\rmlin}^*(x_{t+n})}\\\nonumber
\leq&\tilde{c}^2V(x_t)^2+2\tilde{c}V(x_t)\sqrt{J_{\mathrm{eq}}^{\max}}.
\end{align}
Combining~\eqref{eq:prop_proof_V} and~\eqref{eq:prop2_proof_Jeq_xtxtn}, we conclude that, for $V_{\max},J_{\mathrm{eq}}^{\max}$ sufficiently small, there exists $0<c_{\rmV1}<1$ such that $V(x_{t+n})\leq c_{\rmV1}V(x_t)$.
\end{proof}

\section{Proof of Proposition~\ref{prop:rec_feas2} - candidate 2}\label{appendix_B}
\begin{proof}
In the following, we prove Proposition~\ref{prop:rec_feas2}, i.e., we prove the result in Proposition~\ref{prop:rec_feas} without assuming Inequality~\eqref{eq:prop_case1}.
This is done by showing that the statement remains true if~\eqref{eq:prop_case1} does not hold.
More precisely, we show that there exists $\gamma_1>0$ such that the statement of Proposition~\ref{prop:rec_feas2} holds if 
\begin{align}\label{eq:prop_proof_case21}
\sum_{k=0}^{n-1}\lVert \bar{x}_k^*(t)-x^{\rms*}(t)\rVert_2^2&+\lVert \bar{u}_k^*(t)-u^{\rms*}(t)\rVert_2^2\\\nonumber
&\leq\gamma_1\lVert x^{\rms*}(t)-x^{\rms\rmr}_{\rmlin}(x_t)\rVert_2^2.
\end{align}
Note that~\eqref{eq:prop_proof_case21} implies the existence of some $\tilde{\gamma}>0$ such that
\begin{align}\label{eq:prop_proof_case2}
\sum_{k=0}^{n-1}\lVert \bar{x}_k^*(t)-x^{\rms*}(t)\rVert_2&+\lVert \bar{u}_k^*(t)-u^{\rms*}(t)\rVert_2\\\nonumber
&\leq\tilde{\gamma}\sqrt{\gamma_1}\lVert x^{\rms*}(t)-x^{\rms\rmr}_{\rmlin}(x_t)\rVert_2.
\end{align}
\textbf{(i) Definition of candidate solution}\\
We consider now a different candidate solution at time $t+n$, where the artificial equilibrium input is defined as a convex combination of the optimal artificial equilibrium at time $t$ and the optimal reachable equilibrium input given the system dynamics linearized at $x_t$, i.e.,
\begin{align*}
\hat{u}^\rms(t+n)=\lambda u^{\rms*}(t)+(1-\lambda)u^{\rms\rmr}_{\rmlin}(x_t)
\end{align*}
for some $\lambda\in(0,1)$ which will be fixed later in the proof.
We choose the artificial equilibrium $\hat{x}^\rms(t+n)$ as the corresponding equilibrium state satisfying~\eqref{eq:MPC4} (note that $\hat{x}^\rms(t+n)$ exists due to Assumption~\ref{ass:exist_equil_state}) and the output $\hat{y}^\rms(t+n)$ such that~\eqref{eq:MPC5} holds.
In the following, we show that $x_{t+n}$ is sufficiently close to $\hat{x}^\rms(t+n)$ such that we can steer the system to $\hat{x}^\rms(t+n)$ in $L$ steps.
It follows from the proof of Proposition~\ref{prop:rec_feas} that
\begin{align}\label{eq:prop2_proof_candidate1}
&\lVert x_{t+n}-x^{\rms*}(t)\rVert_2^2\\\nonumber
\stackrel{\eqref{eq:ab_ineq}}{\leq} &2\lVert x_{t+n}-\bar{x}_n^*(t)\rVert_2^2+2\lVert\bar{x}_n^*(t)-x^{\rms*}(t)\rVert_2^2\\\nonumber
\stackrel{\eqref{eq:prop_proof_Lyap_sls},\eqref{eq:prop_proof_xn_xtn}}{\leq}&2\left(2\frac{c_X}{\underline{q}}V(x_t)\sum_{k=0}^{n-2}L_f^k\right)^2+\frac{2}{\underline{q}}V(x_t).
\end{align}
We define $\tilde{x}^\rms\coloneqq\lambda x^{\rms*}(t)+(1-\lambda)x^{\rms\rmr}_{\rmlin}(x_t)$ as the steady-state corresponding to the input $\hat{u}^\rms(t+n)$ for the dynamics linearized at $x_t$.
Note that
\begin{align}\label{eq:prop2_proof_candidate2}
&\lVert x^{\rms*}(t)-\tilde{x}^\rms\rVert_2=(1-\lambda)\lVert x^{\rms*}(t)-x^{\rms\rmr}_{\rmlin}(x_t)\rVert_2\\\nonumber
\leq &\hat{c}_\rml(1-\lambda)\lVert y^{\rms*}(t)-y^{\rms\rmr}_{\rmlin}(x_t)\rVert_2
\end{align}
for some $\hat{c}_\rml>0$ due to the existence of a linear (and hence Lipschitz continuous) map $\hat{g}_{x_t}$ as in~\eqref{eq:ass_unique_steady_state_maps}.
Note that $\hat{x}^\rms(t+n)$ and $\tilde{x}^\rms$ both correspond to the same equilibrium input $\tilde{u}^\rms=\hat{u}^\rms(t+n)$, but to different dynamics linearized at $x_{t+n}$ and $x_t$, respectively.
Therefore, following the same steps as in~\eqref{eq:prop_proof_equil_diff_aux} and~\eqref{eq:prop_proof_equil_diff}, we can derive
\begin{align}\nonumber
&\lVert\tilde{x}^\rms-\hat{x}^\rms(t+n)\rVert_2\stackrel{\eqref{eq:prop_proof_equil_diff_aux}}{\leq}\frac{c_X}{\underline{\sigma}}\left(\lVert\tilde{x}^\rms-x_{t+n}\rVert_2^2+\lVert\tilde{x}^\rms-x_t\rVert_2^2\right)\\\label{eq:prop2_proof_candidate3}
&\stackrel{\eqref{eq:ab_ineq},\eqref{eq:prop2_proof_candidate2}}{\leq}2\frac{c_X}{\underline{\sigma}}\big(2\hat{c}_\rml^2(1-\lambda)^2\lVert y^{\rms*}(t)-y^{\rms\rmr}_{\rmlin}(x_t)\rVert_2^2\\\nonumber
&\quad+\lVert x^{\rms*}(t)-x_{t+n}\rVert_2^2+\lVert x^{\rms*}(t)-x_t\rVert_2^2\big)\\\nonumber
&\stackrel{\eqref{eq:prop_proof_equil_diff}}{\leq} 4\frac{c_X\hat{c}_\rml^2}{\underline{\sigma}}(1-\lambda)^2\lVert y^{\rms*}(t)-y^{\rms\rmr}_{\rmlin}(x_t)\rVert_2^2\\\nonumber
&\quad+2(c_1V(x_t)^2+c_2V(x_t)).
\end{align}
Combining~\eqref{eq:prop2_proof_candidate1}--\eqref{eq:prop2_proof_candidate3}, we see that, if $(1-\lambda)$ and $V_{\max}$ are sufficiently small, then $x_{t+n}$ is arbitrarily close to $\hat{x}^\rms(t+n)$.
Hence, by Assumption~\ref{ass:ctrb}, there exists an input-state trajectory $\hat{u}(t+n),\hat{x}(t+n)$ steering the system from $\hat{x}_0(t+n)=x_{t+n}$ to $\hat{x}_N(t+n)=\hat{x}^\rms(t+n)$ while satisfying $\hat{u}_k(t+n)\in\mathbb{U},k\in\mathbb{I}_{[0,N-1]}$ (note that $\hat{u}^\rms(t+n)\in\text{int}(\mathbb{U})$) and
\begin{align}\nonumber
\sum_{k=0}^{N-1}\lVert\hat{x}_k(t+n)-&\hat{x}^\rms(t+n)\rVert_2+\lVert\hat{u}_k(t+n)-\hat{u}^\rms(t+n)\rVert_2\\\label{eq:prop2_proof_ctrb}
&\leq\Gamma\lVert\hat{x}^\rms(t+n)-x_{t+n}\rVert_2.
\end{align}
\textbf{(ii) Bounds on candidate solution}\\
In the following, we derive multiple bounds on the candidate solution that will be useful in the remainder of the proof.\\
\textbf{(ii.a) Bound on $\lVert\hat{x}^\rms(t+n)-x_{t+n}\rVert_2$}\\
Note that
\begin{align*}
&\lVert\hat{x}^\rms(t+n)-x_{t+n}\rVert_2\\
\leq&\lVert\hat{x}^\rms(t+n)-x^{\rms*}(t)\rVert_2+\lVert x^{\rms*}(t)-x_{t+n}\rVert_2.
\end{align*}
Inequality~\eqref{eq:prop2_proof_candidate1} provides a bound on $\lVert x^{\rms*}(t)-x_{t+n}\rVert_2$.
In the following, we derive a more sophisticated bound which will be required to find a useful bound on $\lVert\hat{x}^\rms(t+n)-x_{t+n}\rVert_2$.
Define $x^\rms\text{$'$}(t+n-1)$ as the steady-state for the linearized dynamics at $x_{t+n-1}$ and with input $u^{\rms*}(t)$.
Then, we have
\begin{align}\label{eq:prop2_proof_dummy1}
&\lVert x^{\rms*}(t)-x_{t+n}\rVert_2\\\nonumber
\leq&\lVert x^{\rms*}(t)-x^\rms\text{$'$}(t+n-1)\rVert_2+\lVert x^\rms\text{$'$}(t+n-1)-x_{t+n}\rVert_2.
\end{align}
In the following, we bound the terms on the right-hand side of~\eqref{eq:prop2_proof_dummy1}.
First, we obtain
\begin{align}\label{eq:prop2_proof_dummy2}
&\quad\lVert x^\rms\text{$'$}(t+n-1)-x_{t+n}\rVert_2\\\nonumber
&=\lVert A_{x_{t+n-1}}(x^\rms\text{$'$}(t+n-1)-x_{t+n-1})\\\nonumber
&+B(u^{\rms*}(t)-u_{t+n-1})\rVert_2\\\nonumber
&\leq\lVert A_{x_{t+n-1}}\rVert_2\lVert x^\rms\text{$'$}(t+n-1)-x_{t+n-1}\rVert_2\\\nonumber
&\quad+\lVert B\rVert_2\lVert u^{\rms*}(t)-u_{t+n-1}\rVert_2\\\nonumber
&\leq\lVert A_{x_{t+n-1}}\rVert_2\Big(\lVert x^\rms\text{$'$}(t+n-1)-x^{\rms*}(t)\rVert_2\\\nonumber
&+\lVert x^{\rms*}(t)-x_{t+n-1}\rVert_2\Big)+\lVert B\rVert_2\lVert u^{\rms*}(t)-\bar{u}_{n-1}^*(t)\rVert_2.
\end{align}
Using the triangle inequality, it holds that
\begin{align}\label{eq:prop2_proof_dummy3}
&\lVert x^{\rms*}(t)-x_{t+n-1}\rVert_2\\\nonumber
\leq&\lVert x^{\rms*}(t)-\bar{x}_{n-1}^*(t)\rVert_2+\lVert \bar{x}_{n-1}^*(t)-x_{t+n-1}\rVert_2.
\end{align}
Defining $c_\rmA\coloneqq\lVert A_{x_{t+n-1}}\rVert_2$, $c_\rmB\coloneqq \lVert B\rVert_2$, $c_{\rmA\rmB}\coloneqq\max\{c_\rmA,c_\rmB\}$ and using~\eqref{eq:prop_proof_case2},~\eqref{eq:prop2_proof_dummy1},~\eqref{eq:prop2_proof_dummy2}, and~\eqref{eq:prop2_proof_dummy3} we have
\begin{align}\nonumber
\lVert x^{\rms*}(t)-x_{t+n}\rVert_2\leq&(1+c_\rmA)\lVert x^{\rms*}(t)-x^\rms\text{$'$}(t+n-1)\rVert_2\\\nonumber
&+c_{\rmA\rmB}\tilde{\gamma}\sqrt{\gamma_1}\lVert x^{\rms*}(t)-x^{\rms\rmr}_{\rmlin}(x_t)\rVert_2\\\label{eq:prop2_proof_dummy4}
&+c_\rmA\lVert\bar{x}_{n-1}^*(t)-x_{t+n-1}\rVert_2.
\end{align}
Next, following the same steps as in Part (ii.c) of the proof of Proposition~\ref{prop:rec_feas} (note that $x^\rms\text{$'$}(t+n-1)$ is defined as an equilibrium of the linearization at $x_{t+n-1}$ with the same input $u^{\rms*}(t)$ as $x^{\rms*}(t)$), it can be shown that
\begin{align}\label{eq:prop2_proof_dummy7}
\lVert x^{\rms*}(t)-x^\rms\text{$'$}(t+n-1)\rVert_2\leq c_1V(x_t)^2+c_2V(x_t).
\end{align}
Moreover, similar to~\eqref{eq:prop_proof_xn_xtn}, it is readily derived that
\begin{align}\label{eq:prop2_proof_dummy8}
\lVert\bar{x}_{n-1}^*(t)-x_{t+n-1}\rVert_2\leq2\frac{c_X}{\underline{q}}V(x_t)\sum_{k=0}^{n-3}L_f^k.
\end{align}
Finally, it follows from~\eqref{eq:prop2_proof_candidate2} and~\eqref{eq:prop2_proof_candidate3} that
\begin{align}\label{eq:prop2_proof_dummy5}
&\lVert\hat{x}^\rms(t+n)- x^{\rms*}(t)\rVert_2\\\nonumber
&\leq \hat{c}_\rml(1-\lambda)\lVert y^{\rms*}(t)-y^{\rms\rmr}_{\rmlin}(x_t)\rVert_2+2(c_1V(x_t)^2+c_2V(x_t))\\\nonumber
&\quad+4\frac{c_X\hat{c}_\rml^2}{\underline{\sigma}}(1-\lambda)^2\lVert y^{\rms*}(t)-y^{\rms\rmr}_{\rmlin}(x_t)\rVert_2^2.
\end{align}
Inserting the bounds~\eqref{eq:prop2_proof_dummy7} and~\eqref{eq:prop2_proof_dummy8} into~\eqref{eq:prop2_proof_dummy4}, and combining this with~\eqref{eq:prop2_proof_dummy5}, we obtain
\begin{align}\label{eq:prop2_proof_hatx_xtn}
&\lVert\hat{x}^\rms(t+n)-x_{t+n}\rVert_2\\\nonumber
\leq&\lVert\hat{x}^\rms(t+n)-x^{\rms*}(t)\rVert_2+\lVert x^{\rms*}(t)-x_{t+n}\rVert_2\\\nonumber
\leq&\hat{c}_\rml(1-\lambda)\lVert y^{\rms*}(t)-y^{\rms\rmr}_{\rmlin}(x_t)\rVert_2+2(c_1V(x_t)^2+c_2V(x_t))\\\nonumber
&+4\frac{c_X\hat{c}_\rml^2}{\underline{\sigma}}(1-\lambda)^2\lVert y^{\rms*}(t)-y^{\rms\rmr}_{\rmlin}(x_t)\rVert_2^2\\\nonumber
&+(1+c_\rmA)(c_1V(x_t)^2+c_2V(x_t))+2c_\rmA\frac{c_X}{\underline{q}}V(x_t)\sum_{k=0}^{n-3}L_f^k\\\nonumber
&+c_{\rmA\rmB}\tilde{\gamma}\sqrt{\gamma_1}\lVert x^{\rms*}(t)-x^{\rms\rmr}_{\rmlin}(x_t)\rVert_2\\\nonumber
\stackrel{\eqref{eq:prop2_proof_candidate2}}{\leq}& \hat{c}_\rml((1-\lambda)+c_{\rmA\rmB}\tilde{\gamma}\sqrt{\gamma_1})\lVert y^{\rms*}(t)-y^{\rms\rmr}_{\rmlin}(x_t)\rVert_2+c_4V(x_t)\\\nonumber
&+4\frac{c_X\hat{c}_\rml^2}{\underline{\sigma}}(1-\lambda)^2\lVert y^{\rms*}(t)-y^{\rms\rmr}_{\rmlin}(x_t)\rVert_2^2
\end{align}
for some $c_4>0$, using that $V(x_t)\leq V_{\max}$ in the last inequality.\\
\textbf{(ii.b) Bound on $\lVert\hat{y}^\rms(t+n)-y^\rmr\rVert_S^2-\lVert y^{\rms*}(t)-y^\rmr\rVert_S^2$}\\
Similar to~\cite{koehler2020nonlinear,berberich2020tracking}, it is straightforward to exploit the convexity condition~\eqref{eq:ass_strongly_convex} in order to derive
\begin{align}\label{eq:prop2_proof_iib1}
&\lVert\tilde{y}^\rms-y^\rmr\rVert_S^2-\lVert y^{\rms*}(t)-y^\rmr\rVert_S^2\\\nonumber
\leq&-(1-\lambda^2)\lVert y^{\rms*}(t)-y^{\rms\rmr}_{\rmlin}(x_t)\rVert_S^2,
\end{align}
where $\tilde{y}^\rms=\lambda y^{\rms*}(t)+(1-\lambda)y^{\rms\rmr}_{\rmlin}(x_t)$.
Moreover,~\eqref{eq:prop_proof_std_norm_ineq} implies
\begin{align}\label{eq:prop2_proof_iib2}
&\lVert\hat{y}^\rms(t+n)-y^\rmr\rVert_S^2-\lVert \tilde{y}^\rms-y^\rmr\rVert_S^2\\\nonumber
\leq&\lVert\hat{y}^\rms(t+n)-\tilde{y}^\rms\rVert_S^2+2\lVert\hat{y}^\rms(t+n)-\tilde{y}^\rms\rVert_S\lVert\tilde{y}^\rms-y^\rmr\rVert_S.
\end{align}
Recall that $\tilde{u}^\rms=\lambda u^{\rms*}(t)+(1-\lambda)u^{\rms\rmr}_{\rmlin}(x_t)=\hat{u}^\rms(t+n)$.
Hence, using an inequality of the form~\eqref{eq:diff_lin} for the vector field $h$, there exist constants $c_{Xh},L_h\geq0$ such that
\begin{align}\label{eq:prop2_proof_iib7}
&\lVert\hat{y}^\rms(t+n)-\tilde{y}^\rms\rVert_2\\\nonumber
=&\lVert h_{x_{t+n}}(\hat{x}^\rms(t+n),\tilde{u}^\rms)-h_{x_t}(\tilde{x}^\rms,\tilde{u}^\rms)\rVert_2\\\nonumber
\leq&c_{Xh}\lVert\hat{x}^\rms(t+n)-x_{t+n}\rVert_2^2+c_{Xh}\lVert\tilde{x}^\rms-x_t\rVert_2^2\\\nonumber
&+L_h\lVert\hat{x}^\rms(t+n)-\tilde{x}^\rms\rVert_2.
\end{align}
The second term on the right-hand side is bounded as
\begin{align*}
\lVert\tilde{x}^\rms-x_t\rVert_2^2\stackrel{\eqref{eq:ab_ineq}}{\leq}&2\lVert\tilde{x}^\rms-x^{\rms*}(t)\rVert_2^2+2\lVert x^{\rms*}(t)-x_t\rVert_2^2\\
\stackrel{\eqref{eq:prop_proof_Lyap_sls},\eqref{eq:prop2_proof_candidate2}}{\leq}&2\hat{c}_\rml^2(1-\lambda)^2\lVert y^{\rms*}(t)-y^{\rms\rmr}_{\rmlin}(x_t)\rVert_2^2+\frac{2}{\underline{q}}V(x_t).
\end{align*}
Using in addition the bounds \eqref{eq:prop2_proof_candidate3} and~\eqref{eq:prop2_proof_hatx_xtn}, this implies
\begin{align}\label{eq:prop2_proof_iib3}
&\lVert\hat{y}^\rms(t+n)-\tilde{y}^\rms\rVert_2\\\nonumber
&\leq c_{Xh}\Big(\hat{c}_\rml((1-\lambda)+c_{\rmA\rmB}\tilde{\gamma}\sqrt{\gamma_1})\lVert y^{\rms*}(t)-y^{\rms\rmr}_{\rmlin}(x_t)\rVert_2\\\nonumber
&\quad+c_4V(x_t)+4\frac{c_X\hat{c}_\rml^2}{\underline{\sigma}}(1-\lambda)^2\lVert y^{\rms*}(t)-y^{\rms\rmr}_{\rmlin}(x_t)\rVert_2^2\Big)^2\\\nonumber
&\quad+c_{Xh}\left(2\hat{c}_\rml^2(1-\lambda)^2\lVert y^{\rms*}(t)-y^{\rms\rmr}_{\rmlin}(x_t)\rVert_2^2+\frac{2}{\underline{q}}V(x_t)\right)\\\nonumber
&\quad+4L_h\frac{c_X\hat{c}_\rml^2}{\underline{\sigma}}(1-\lambda)^2\lVert y^{\rms*}(t)-y^{\rms\rmr}_{\rmlin}(x_t)\rVert_2^2\\\nonumber
&\quad+2L_h(c_1V(x_t)^2+c_2V(x_t)).
\end{align}
Further, by convexity (recall that $\tilde{y}^\rms$ is defined as a convex combination of $y^{\rms*}(t)$ and $y^{\rms\rmr}_{\rmlin}(x_t)$) we have
\begin{align}\label{eq:prop2_proof_iib4}
&\lVert\tilde{y}^\rms-y^\rmr\rVert_S^2\\\nonumber
\leq&\lambda\lVert y^{\rms*}(t)-y^\rmr\rVert_S^2+(1-\lambda)\lVert y^{\rms\rmr}_{\rmlin}(x_t)-y^\rmr\rVert_S^2\\\nonumber
\leq&\lambda(V(x_t)+J_{\mathrm{eq}}^{\max})+(1-\lambda)J_{\mathrm{eq}}^{\max}.
\end{align}
The bound~\eqref{eq:prop2_proof_iib2} together with the subsequently derived bounds will play an important role in the remainder of the proof.
To this end, using~\eqref{eq:prop2_proof_iib1}, we conclude
\begin{align}\label{eq:prop2_proof_final_dummy1}
&\lVert\hat{y}^\rms(t+n)-y^\rmr\rVert_S^2-\lVert y^{\rms*}(t)-y^\rmr\rVert_S^2\\\nonumber
=&\lVert\hat{y}^\rms(t+n)-y^\rmr\rVert_S^2-\lVert\tilde{y}^\rms-y^\rmr\rVert_S^2\\\nonumber
&+\lVert\tilde{y}^\rms-y^\rmr\rVert_S^2-\lVert y^{\rms*}(t)-y^\rmr\rVert_S^2\\\nonumber
\stackrel{\eqref{eq:prop2_proof_iib1}}{\leq}&-(1-\lambda^2)\lVert y^{\rms*}(t)-y^{\rms\rmr}_{\rmlin}(x_t)\rVert_S^2\\\nonumber
&+\lVert\hat{y}^\rms(t+n)-y^\rmr\rVert_S^2-\lVert\tilde{y}^\rms-y^\rmr\rVert_S^2.
\end{align}
\textbf{(iii) Invariance of $V(x_t)\leq V_{\max}$}\\
It follows directly from~\eqref{eq:prop2_proof_ctrb},~\eqref{eq:prop2_proof_hatx_xtn}, and~\eqref{eq:prop2_proof_final_dummy1}, and using the inequality $a^2+b^2\leq(a+b)^2$ for $a,b\geq0$, that
\begin{align*}
&J_N^*(x_{t+n})-J_N^*(x_t)\\
&\leq\lambda_{\max}(Q,R)\Gamma^2\Big(4\frac{c_X\hat{c}_\rml^2}{\underline{\sigma}}(1-\lambda)^2\lVert y^{\rms*}(t)-y^{\rms\rmr}_{\rmlin}(x_t)\rVert_2^2\\
&+\hat{c}_\rml((1-\lambda)+c_{\rmA\rmB}\tilde{\gamma}\sqrt{\gamma_1})\lVert y^{\rms*}(t)-y^{\rms\rmr}_{\rmlin}(x_t)\rVert_2+c_4V(x_t)\Big)^2\\
&-\lVert x_t-x^{\rms*}(t)\rVert_Q^2-(1-\lambda^2)\lVert y^{\rms*}(t)-y^{\rms\rmr}_{\rmlin}(x_t)\rVert_S^2\\
&+\lVert\hat{y}^\rms(t+n)-y^\rmr\rVert_S^2-\lVert\tilde{y}^\rms-y^\rmr\rVert_S^2.
\end{align*}
If $V_{\max}$, $J_{\mathrm{eq}}^{\max}$, $\gamma_1$ and $(1-\lambda)$ are all sufficiently small, then using the bounds derived in Part (ii.b) of the proof, it follows directly that
\begin{align*}
&J_N^*(x_{t+n})-J_N^*(x_t)\\
\leq&-\lVert x_t-x^{\rms*}(t)\rVert_Q^2-c_5\lVert y^{\rms*}(t)-y^{\rms\rmr}_{\rmlin}(x_t)\rVert_S^2+c_6V(x_t)
\end{align*}
with some $c_5,c_6>0$, where $c_6$ becomes arbitrarily small if $V_{\max}$ and $J_{\mathrm{eq}}^{\max}$ are sufficiently small.
It follows from the existence of a linear map $\hat{g}_x$ as in~\eqref{eq:ass_unique_steady_state_maps} that $\lVert y^{\rms*}(t)-y^{\rms\rmr}_{\rmlin}(x_t)\rVert_S^2\geq \frac{1}{\hat{c}_\rml}\lVert x^{\rms*}(t)-x^{\rms\rmr}_{\rmlin}(x_t)\rVert_S^2$ with $\hat{c}_\rml>0$.
Hence, using the upper bound~\eqref{eq:lem_value_fcn_upper_bound}, we obtain
\begin{align*}
&J_N^*(x_{t+n})-J_N^*(x_t)\\
\leq&-\frac{\min\{\underline{q},\frac{c_5}{\hat{c}_\rml}\underline{s}\}}{2}\lVert x_t-x^{\rms\rmr}_{\rmlin}(x_t)\rVert_2^2+c_6V(x_t)\\
\stackrel{\eqref{eq:lem_value_fcn_upper_bound}}{\leq}&(\tilde{c}_{\rmV2}-1)V(x_t)
\end{align*}
for some $0<\tilde{c}_{\rmV2}<1$, assuming that $V_{\max}$ and $J_{\mathrm{eq}}^{\max}$ and hence $c_6$ are sufficiently small.
This leads to
\begin{align*}
V(x_{t+n})\leq \tilde{c}_{\rmV2}V(x_t)+J_{\mathrm{eq},\rmlin}^*(x_t)-J_{\mathrm{eq},\rmlin}^*(x_{t+n}).
\end{align*}
Finally, following the same steps as in the proof of Proposition~\ref{prop:rec_feas}, we can show that this implies the existence of a constant $0<c_{\rmV2}<1$ such that $V(x_{t+n})\leq c_{\rmV2} V(x_t)$.
Combining this with the statement of Proposition~\ref{prop:rec_feas}, we obtain $V(x_{t+n})\leq c_\rmV V(x_t)$ for $c_\rmV\coloneqq\max\{c_{\rmV1},c_{\rmV2}\}<1$.
\end{proof}

\section{Sufficient conditions for Assumption~\ref{ass:convex_steady_state_manifold}}\label{app:cond_ass}

In the following, we present sufficient conditions for Assumption~\ref{ass:convex_steady_state_manifold}.
Throughout this section, we assume that $m=p$, i.e., the numbers of inputs and outputs coincide.
Further, we assume that the target setpoint $y^\rmr$ is reachable, as captured in the following assumption.
\begin{assumption}\label{ass:reachability}
(Reachability)
For any $\tilde{x}\in\mathbb{R}^n$, the target setpoint $y^\rmr$ is reachable under the linearized dynamics, i.e., $y^{\rms\rmr}_{\rmlin}(\tilde{x})=y^\rmr$.
\end{assumption}
Assumption~\ref{ass:reachability} means that the optimal reachable output equilibrium $y^{\rms\rmr}_{\rmlin}(\tilde{x})$ at any linearization point $\tilde{x}\in\mathbb{R}^n$ is equal to the target setpoint $y^\rmr$ and hence, inserting $\tilde{x}=x^{\rms\rmr}$, the same holds true for the nonlinear optimal reachable output equilibrium, i.e, $y^{\rms\rmr}=y^\rmr$.
Assumption~\ref{ass:reachability} is only restrictive if the equilibrium input leading to the target setpoint $y^\rmr$ does not satisfy the input constraints.
In particular, Assumption~\ref{ass:reachability} always holds in case of no input constraints, i.e., if $\mathbb{U}=\mathbb{R}^m$, due to Assumption~\ref{ass:unique_steady_state} and $m=p$.
Before verifying Assumption~\ref{ass:convex_steady_state_manifold}, we first prove a technical intermediate result.
\begin{lemma}\label{lem:app_1}
Suppose Assumptions~\ref{ass:C2},~\ref{ass:unique_steady_state},~\ref{ass:exist_equil_state}, and~\ref{ass:steady_state_compact} hold and $m=p$.
Then, for any compact set $X$ with $\mathcal{B}\subseteq X\times\mathbb{U}$ (cf. Assumption~\ref{ass:steady_state_compact}), there exists a constant $c_{\rms1}>0$ such that for any steady-state $(x^\rms,u^\rms)\in\calZ^\rms$ of the nonlinear system, any state $\tilde{x}\in X$, and any input $\tilde{u}\in\mathbb{U}$, it holds that
\begin{align}\label{eq:lem_app_1}
&\lVert (x^\rms,u^\rms)-(\tilde{x},\tilde{u})\rVert_2\\\nonumber
\leq &c_{\rms1}\big(\lVert h(x^\rms,u^\rms)-h(\tilde{x},\tilde{u})\rVert_2+\lVert \tilde{x}-f(\tilde{x},\tilde{u})\rVert_2\big).
\end{align}
\end{lemma}
\begin{proof}
Define the map $s:X\times\mathbb{U}\to X\times h(X,\mathbb{U})$ with
\begin{align}
s(x,u)=\begin{bmatrix}x-f(x,u)\\h(x,u)\end{bmatrix}.
\end{align}
First, recall that, by Assumption~\ref{ass:unique_steady_state}, $s(x,u)$ is invertible with smooth and hence (on the compact set $h(X,\mathbb{U})$) Lipschitz continuous inverse, cf.~\cite[Condition (1.1)]{radulescu1980global}.
Thus, there exists $c_{\rms1}>0$ such that, for any $x_1,x_2\in X$, $u_1,u_2\in\mathbb{U}$, it holds that
\begin{align}\label{eq:lem_app_1_proof1}
\lVert (x_1,u_1)-(x_2,u_2)\rVert_2&\leq c_{\rms1}\lVert s(x_1,u_1)-s(x_2,u_2)\rVert_2.
\end{align}
Choosing $(x^\rms,u^\rms)\in\calZ^\rms\subseteq\mathcal{B}\times\mathbb{U}$, $\tilde{x}\in X$, $\tilde{u}\in\mathbb{U}$,~\eqref{eq:lem_app_1_proof1} implies
\begin{align}
\lVert (x^\rms,u^\rms)-(\tilde{x},\tilde{u})\rVert_2
&\leq c_{\rms1}\lVert s(x^\rms,u^\rms)-s(\tilde{x},\tilde{u})\rVert_2\\\nonumber
&\leq c_{\rms1}\big(\lVert x^\rms-f(x^\rms,u^\rms)-\tilde{x}+f(\tilde{x},\tilde{u})\rVert_2\\\nonumber
&\quad+\lVert h(x^\rms,u^\rms)-h(\tilde{x},\tilde{u})\rVert_2\big).
\end{align}
Using $x^\rms=f(x^\rms,u^\rms)$, we infer~\eqref{eq:lem_app_1}.
\end{proof}

Let us now prove~\eqref{eq:convexity_nonlinear} based on Lemma~\ref{lem:app_1} and the given assumptions.

\begin{proposition}\label{prop:convexity_nonlinear}
If Assumptions~\ref{ass:C2},~\ref{ass:unique_steady_state},~\ref{ass:exist_equil_state},~\ref{ass:steady_state_compact}, and~\ref{ass:reachability} hold and $m=p$, then Assumption~\ref{ass:convex_steady_state_manifold} holds.
%
\end{proposition}
\begin{proof}
\textbf{Proof of $c_{\mathrm{eq},1}\lVert\hat{x}-x^{\rms\rmr}_{\rmlin}(\hat{x})\rVert_2^2\leq\lVert\hat{x}-x^{\rms\rmr}\rVert_2^2$}\\
The linear map in~\eqref{eq:ass_unique_steady_state_maps} can be written explicitly as 
\begin{align}\label{eq:prop_convexity_nonlinear_proof0}
\hat{g}_{\hat{x}}(y^\rms)=\underbrace{\begin{bmatrix}A_{\hat{x}}-I&B\\C_{\hat{x}}&D\end{bmatrix}^{-1}}_{M_{\hat{x}}^{-1}\coloneqq}\begin{bmatrix}-e_{\hat{x}}\\y^\rms-r_{\hat{x}}\end{bmatrix}
\end{align}
for any $y^\rms\in\mathcal{Z}_{\rmy,\rmlin}^\rms(\hat{x})$.
In the following, we derive a bound on the difference $M_{\hat{x}}^{-1}-M_{x^{\rms\rmr}}^{-1}$ which we then use to obtain the desired statement.
To this end, it is readily derived that
\begin{align*}
\lVert M_{\hat{x}}^{-1}-M_{x^{\rms\rmr}}^{-1}\rVert_2
&\leq\lVert M_{\hat{x}}^{-1}\rVert_2\lVert I-M_{\hat{x}}M_{x^{\rms\rmr}}^{-1}\rVert_2,\\
\lVert I-M_{\hat{x}}M_{x^{\rms\rmr}}^{-1}\rVert_2&\leq \lVert M_{x^{\rms\rmr}}-M_{\hat{x}}\rVert_2
\lVert M_{x^{\rms\rmr}}^{-1}\rVert_2.
\end{align*}
Combining these inequalities and using that $\lVert M_{\hat{x}}^{-1}\rVert_2\leq\frac{1}{\sigma_\rms}$ and $\lVert M_{x^{\rms\rmr}}^{-1}\rVert_2\leq\frac{1}{\sigma_\rms}$ due to Assumption~\ref{ass:unique_steady_state}, we obtain
\begin{align}\label{eq:prop_convexity_nonlinear_proof1}
\lVert M_{\hat{x}}^{-1}-M_{x^{\rms\rmr}}^{-1}\rVert_2\leq\frac{1}{{\sigma_\rms}^2}\lVert M_{x^{\rms\rmr}}-M_{\hat{x}}\rVert_2.
\end{align}
Similar to~\eqref{eq:diff_lin}, it holds that
\begin{align}\label{eq:prop_convexity_nonlinear_proof2}
\lVert M_{x^{\rms\rmr}}-M_{\hat{x}}\rVert_2\leq \tilde{c}_X\lVert x^{\rms\rmr}-\hat{x}\rVert_2^2
\end{align}
for some $\tilde{c}_X>0$.
Here, we use that $\hat{x}\in X$ by assumption and $x^{\rms\rmr}\in\mathcal{B}_x\subseteq X$ due to Assumption~\ref{ass:steady_state_compact}, where $\mathcal{B}_x$ denotes the projection of $\mathcal{B}$ on the state component.
To summarize, combining~\eqref{eq:prop_convexity_nonlinear_proof1} and~\eqref{eq:prop_convexity_nonlinear_proof2} and using that $y^{\rms\rmr}=y^{\rms\rmr}_{\rmlin}(\hat{x})=y^\rmr$ by Assumption~\ref{ass:reachability}, we have
\begin{align}\label{eq:prop_convexity_nonlinear_proof3}
&\lVert x^{\rms\rmr}_{\rmlin}(\hat{x})-x^{\rms\rmr}\rVert_2\\\nonumber
\stackrel{\eqref{eq:prop_convexity_nonlinear_proof0}}{\leq}&\left\lVert M_{\hat{x}}^{-1}\begin{bmatrix}-e_{\hat{x}}\\y^{\rms\rmr}-r_{\hat{x}}\end{bmatrix}
-M_{x^{\rms\rmr}}^{-1}\begin{bmatrix}-e_{x^{\rms\rmr}}\\y^{\rms\rmr}-r_{x^{\rms\rmr}}\end{bmatrix}\right\rVert_2\\\nonumber
\leq&\left\lVert (M_{\hat{x}}^{-1}-M_{x^{\rms\rmr}}^{-1})\begin{bmatrix}-e_{x^{\rms\rmr}}\\y^{\rms\rmr}-r_{x^{\rms\rmr}}
\end{bmatrix}\right\rVert_2
+\left\lVert M_{\hat{x}}^{-1}\begin{bmatrix}e_{x^{\rms\rmr}}-e_{\hat{x}}\\r_{x^{\rms\rmr}}-r_{\hat{x}}
\end{bmatrix}\right\rVert_2\\\nonumber
\leq&\frac{\tilde{c}_X}{\sigma_\rms^2}\left\lVert\begin{bmatrix}-e_{x^{\rms\rmr}}\\y^{\rms\rmr}-r_{x^{\rms\rmr}}
\end{bmatrix}\right\rVert_2\lVert x^{\rms\rmr}-\hat{x}\rVert_2^2\\\nonumber
&+\lVert M_{\hat{x}}^{-1}\rVert_2\left\lVert\begin{bmatrix}e_{\hat{x}}-e_{x^{\rms\rmr}}\\r_{\hat{x}}-r_{x^{\rms\rmr}}\end{bmatrix}\right\rVert_2.
\end{align}
Further, using a similar argument as in Proposition~\ref{prop:Lin} for the vector fields $f_0$, $h_0$, there exists a constant $c_{X0}>0$ such that
\begin{align*}
\lVert M_{\hat{x}}^{-1}\rVert_2\left\lVert\begin{bmatrix}e_{\hat{x}}-e_{x^{\rms\rmr}}\\r_{\hat{x}}-r_{x^{\rms\rmr}}\end{bmatrix}\right\rVert_2
\leq&\frac{1}{\sigma_\rms}c_{X0}\lVert \hat{x}-x^{\rms\rmr}\rVert_2^2.
\end{align*}
Together with~\eqref{eq:prop_convexity_nonlinear_proof3}, this implies
\begin{align*}
&\lVert x^{\rms\rmr}_{\rmlin}(\hat{x})-x^{\rms\rmr}\rVert_2\leq \bar{c}_{\mathrm{eq},1}\lVert x^{\rms\rmr}-\hat{x}\rVert_2
\end{align*}
with
\begin{align*}
\bar{c}_{\mathrm{eq},1}\coloneqq &\left(\frac{\tilde{c}_X}{{\sigma_\rms}^2}
\left\lVert\begin{bmatrix}-e_{x^{\rms\rmr}}\\y^{\rms\rmr}-r_{x^{\rms\rmr}}
\end{bmatrix}\right\rVert_2
+\frac{c_{X0}}{\sigma_\rms}\right)\max_{\hat{x}\in X,x^\rms\in\mathcal{B}_x}\lVert x^\rms-\hat{x}\rVert_2.
\end{align*}
Finally, using
\begin{align*}
\lVert\hat{x}-x^{\rms\rmr}_{\rmlin}(\hat{x})\rVert_2\leq&\lVert\hat{x}-x^{\rms\rmr}\rVert_2+\lVert x^{\rms\rmr}-x^{\rms\rmr}_{\rmlin}(\hat{x})\rVert_2,
\end{align*}
we obtain the left inequality in~\eqref{eq:convexity_nonlinear} for $c_{\mathrm{eq},1}\coloneqq\frac{1}{(1+\bar{c}_{\mathrm{eq},1})^2}>0$.\\
\textbf{Proof of $\lVert\hat{x}-x^{\rms\rmr}\rVert_2^2\leq c_{\mathrm{eq},2}\lVert\hat{x}-x^{\rms\rmr}_{\rmlin}(\hat{x})\rVert_2^2$}\\
Applying Lemma~\ref{lem:app_1} with $(x^\rms,u^\rms)=(x^{\rms\rmr},u^{\rms\rmr})$, $(\tilde{x},\tilde{u})=(x^{\rms\rmr}_{\rmlin}(\hat{x}),u^{\rms\rmr}_{\rmlin}(\hat{x}))$, we obtain
\begin{align}\label{eq:prop_convexity_nonlinear_proof4}
\lVert x^{\rms\rmr}-x^{\rms\rmr}_{\rmlin}(\hat{x})\rVert_2
&\leq c_{\rms1}\big(\lVert y^{\rms\rmr}-h(x^{\rms\rmr}_{\rmlin}(\hat{x}),u^{\rms\rmr}_{\rmlin}(\hat{x}))\rVert_2\\\nonumber
&\quad+\lVert x^{\rms\rmr}_{\rmlin}(\hat{x})-f(x^{\rms\rmr}_{\rmlin}(\hat{x}),u^{\rms\rmr}_{\rmlin}(\hat{x}))\rVert_2\big).
\end{align}
Using $y^{\rms\rmr}=h_{\hat{x}}(x^{\rms\rmr}_{\rmlin}(\hat{x}),u^{\rms\rmr}_{\rmlin}(\hat{x}))$ due to Assumption~\ref{ass:reachability} together with a bound of the form~\eqref{eq:diff_lin} for the vector field $h$, the first term is bounded as
\begin{align*}
&\lVert h_{\hat{x}}(x^{\rms\rmr}_{\rmlin}(\hat{x}),u^{\rms\rmr}_{\rmlin}(\hat{x}))-h_{x^{\rms\rmr}_{\rmlin}(\hat{x})}(x^{\rms\rmr}_{\rmlin}(\hat{x}),u^{\rms\rmr}_{\rmlin}(\hat{x}))\rVert_2\\
\leq&c_{Xh}\lVert \hat{x}-x^{\rms\rmr}_{\rmlin}(\hat{x})\rVert_2^2
\end{align*}
for some $c_{Xh}>0$.
Moreover, the second term on the right-hand side of~\eqref{eq:prop_convexity_nonlinear_proof4} is bounded as
\begin{align*}
&\lVert x^{\rms\rmr}_{\rmlin}(\hat{x})-f(x^{\rms\rmr}_{\rmlin}(\hat{x}),u^{\rms\rmr}_{\rmlin}(\hat{x}))\rVert_2\\
=&\lVert f_{\hat{x}}(x^{\rms\rmr}_{\rmlin}(\hat{x}),u^{\rms\rmr}_{\rmlin}(\hat{x}))-f_{x^{\rms\rmr}_{\rmlin}(\hat{x})}(x^{\rms\rmr}_{\rmlin}(\hat{x}),u^{\rms\rmr}_{\rmlin}(\hat{x}))\rVert_2\\
\stackrel{\eqref{eq:diff_lin}}{\leq}&c_X\lVert \hat{x}-x^{\rms\rmr}_{\rmlin}(\hat{x})\rVert_2^2.
\end{align*}
Combining the above inequalities, we obtain
\begin{align*}
&\lVert \hat{x}-x^{\rms\rmr}\rVert_2\\
\leq&\lVert \hat{x}-x^{\rms\rmr}_{\rmlin}(\hat{x})\rVert_2+\lVert x^{\rms\rmr}_{\rmlin}(\hat{x})-x^{\rms\rmr}\rVert_2\\
\leq&\lVert \hat{x}-x^{\rms\rmr}_{\rmlin}(\hat{x})\rVert_2+c_{\rms1}(c_{Xh}+c_X)\lVert\hat{x}-x^{\rms\rmr}_{\rmlin}(\hat{x})\rVert_2^2,
\end{align*}
leading to the right inequality in~\eqref{eq:convexity_nonlinear} for 
\begin{align*}
c_{\mathrm{eq},2}\coloneqq (1+c_{\rms1}(c_{Xh}+c_X)\max_{\hat{x}\in X}\lVert \hat{x}-x^{\rms\rmr}_{\rmlin}(\hat{x})\rVert_2)^2.
\end{align*}
\end{proof}

\begin{IEEEbiography}[{\includegraphics[width=1in,height=1.25in,clip,keepaspectratio]{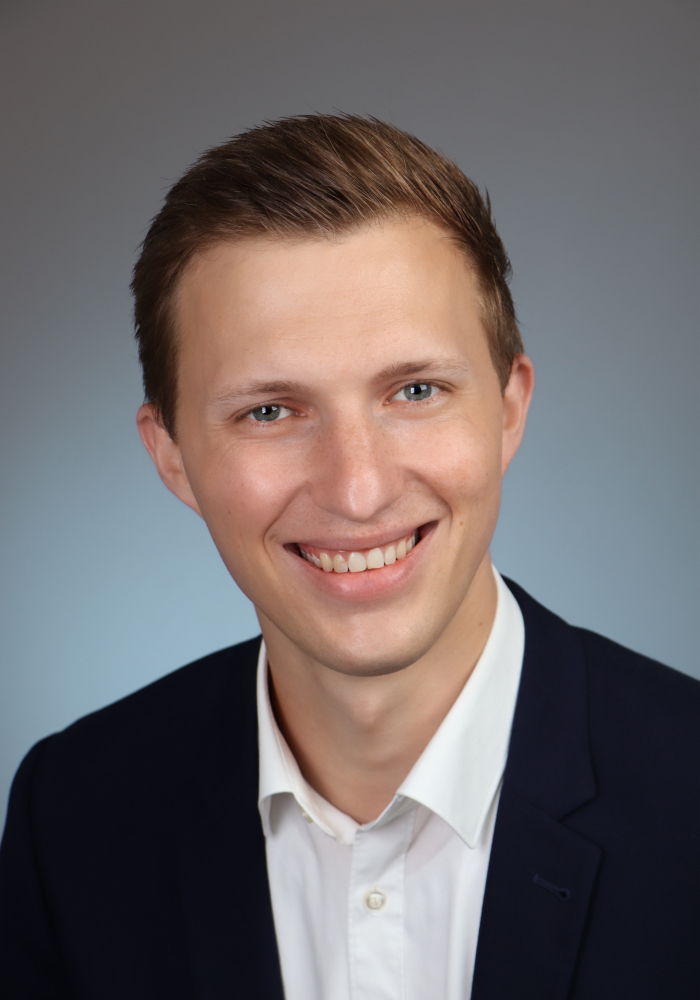}}]{Julian Berberich}
received the Master’s degree in Engineering Cybernetics from the University of Stuttgart, Germany, in 2018. Since 2018, he has been a Ph.D. student at the Institute for Systems Theory and Automatic Control under supervision of Prof. Frank Allg\"ower and a member of the International Max-Planck Research School (IMPRS) at the University of Stuttgart. He has received the Outstanding Student Paper Award at the 59th Conference on Decision and Control in 2020. His research interests are in the area of data-driven analysis and control.
\end{IEEEbiography}

\begin{IEEEbiography}[{\includegraphics[width=1in,height=1.25in,clip,keepaspectratio]{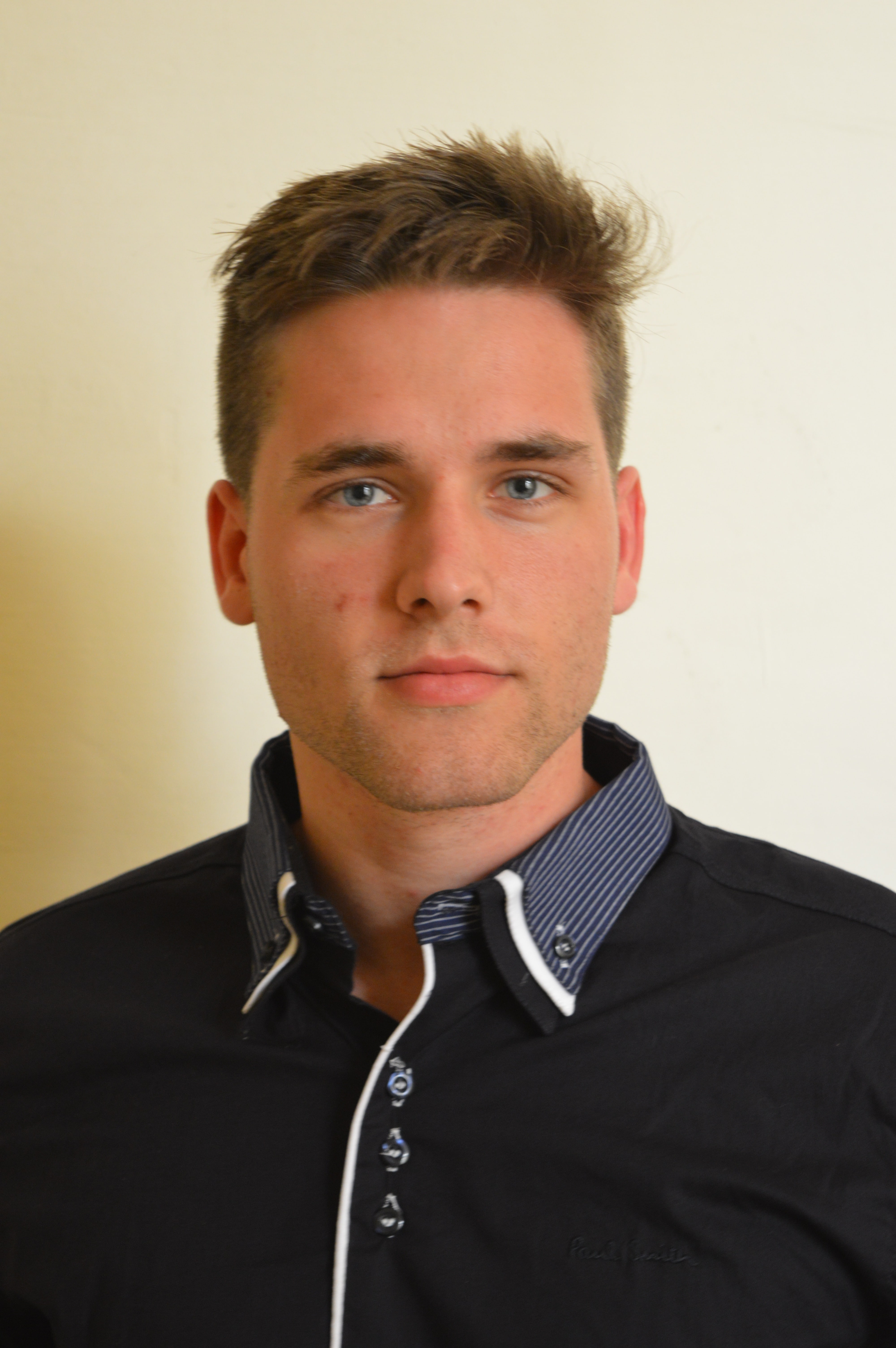}}]{Johannes K\"ohler}
received his Master degree in Engineering Cybernetics from the University of Stuttgart, Germany, in 2017.
In 2021, he obtained a Ph.D. in Mechanical Engineering, also from the University of Stuttgart, Germany. 
Since then, he is a postdoctoral researcher at the \textit{Institute for Dynamic Systems and Control} at ETH Zürich. 
His research interests are in the area of model predictive control and the control of nonlinear uncertain systems.
 \end{IEEEbiography}

\begin{IEEEbiography}[{\includegraphics[width=1in,height=1.25in,clip,keepaspectratio]{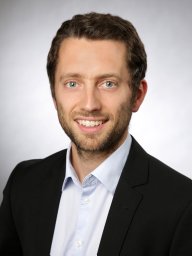}}]{Matthias A. M\"uller}
received a Diploma degree in Engineering Cybernetics from the University of Stuttgart, Germany, and an M.S. in Electrical and Computer Engineering from the University of Illinois at Urbana-Champaign, US, both in 2009.
In 2014, he obtained a Ph.D. in Mechanical Engineering, also from the University of Stuttgart, Germany, for which he received the 2015 European Ph.D. award on control for complex and heterogeneous systems. Since 2019, he is director of the Institute of Automatic Control and full professor at the Leibniz University Hannover, Germany. 
He obtained an ERC Starting Grant in 2020 and is recipient of the inaugural Brockett-Willems Outstanding Paper Award for the best paper published in Systems \& Control Letters in the period 2014-2018. His research interests include nonlinear control and estimation, model predictive control, and data-/learning-based control, with application in different fields including biomedical engineering.
\end{IEEEbiography}

\begin{IEEEbiography}[{\includegraphics[width=1in,height=1.25in,clip,keepaspectratio]{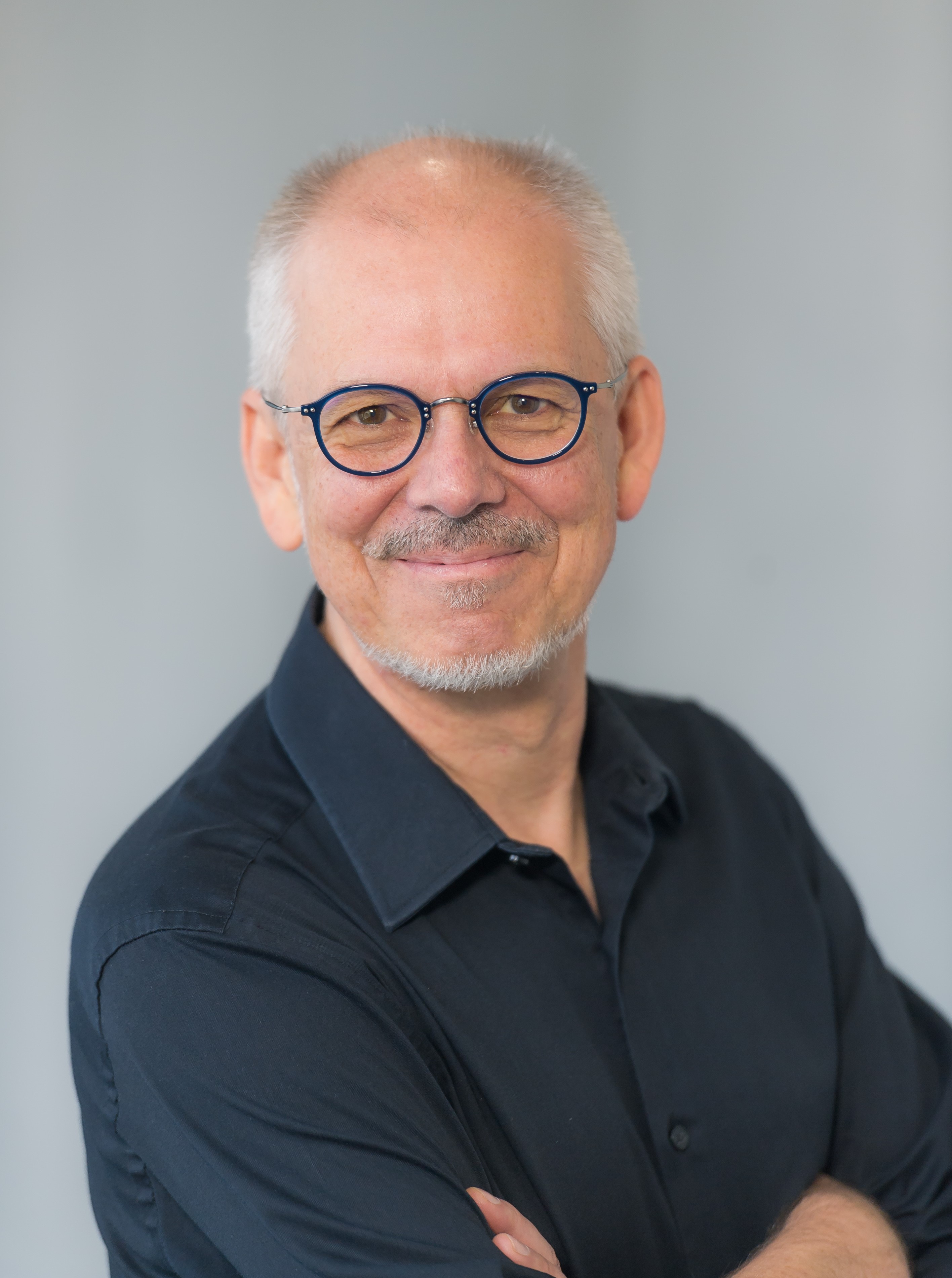}}]{Frank Allg\"ower}
is professor of mechanical engineering at the University of Stuttgart, Germany, and Director of the Institute for Systems Theory and Automatic Control (IST) there.\\ 
Frank is active in serving the community in several roles: Among others he has been President of the International Federation of Automatic Control (IFAC) for the years 2017-2020, Vice-president for Technical Activities of the IEEE Control Systems Society for 2013/14, and Editor of the journal Automatica from 2001 until 2015. From 2012 until 2020 Frank served in addition as Vice-president for the German Research Foundation (DFG), which is Germany’s most important research funding organization. \\
His research interests include predictive control, data-based control, networked control, cooperative control, and nonlinear control with application to a wide range of fields including systems biology.
\end{IEEEbiography}
\end{document}